\documentclass{article}    
\usepackage{graphics}
\usepackage{graphicx}
\usepackage{epsfig}
\usepackage[numbers,sort&compress]{natbib}
\usepackage{anysize}
\usepackage{array}
\usepackage{authblk}
\usepackage[fleqn]{amsmath}
\usepackage{amssymb}
\usepackage{latexsym,amssymb,amscd,amsthm,amsxtra,mathrsfs}
\usepackage{mathtools}
\usepackage{subfigure}
\usepackage{tabularx}
\usepackage{enumerate}
\usepackage{multirow}
\usepackage{float}
\usepackage{placeins}
\usepackage{times}
\usepackage{soul,xcolor}
\usepackage{algorithm}
\usepackage{algpseudocode}
\usepackage{makecell, booktabs, caption}
\usepackage[colorlinks = true,
            linkcolor = blue,
            urlcolor  = blue,
            citecolor = blue,
            anchorcolor = blue]{hyperref}
\definecolor{lightgray}{gray}{0.9}
\newtheorem{theorem}{Theorem}
\newtheorem{definition}{Definition}
\newtheorem{lemma}{Lemma}[section]
\newtheorem{proposition}{Proposition}[section]
\newtheorem{remark}{Remark}[section]
\newtheorem{example}{Example}[section]

\numberwithin{equation}{section}

\graphicspath{ {figures/} }
\allowbreak
\begin{document}
\title {\textbf{Dirichlet Extremals for Discrete Plateau Problems in GT–Bézier Spaces via PSO}}
\author[1]{Muhammad Ammad\thanks{Corresponding author: \url{21481199@life.hkbu.edu.hk}}}
\author[2]{Md Yushalify Misro\thanks{E-mail: \url{yushalify@usm.my}}}
\author[2]{Samia Bibi}
\author[2]{Ahmad Ramli}

\affil[1]{\small Department of Mathematics, Hong Kong Baptist University, Kowloon Tong, Hong Kong}
\affil[2]{\small School of Mathematical Sciences, Universiti Sains Malaysia, 11800 Gelugor, Pulau Pinang, Malaysia}
\date{}
\maketitle
\vspace{-0.5cm}

\begin{abstract}
We study a discrete analogue of the parametric Plateau problem in a non-polynomial tensor-product surface spaces generated by the generalized trigonometric (GT)--B\'ezier basis. Boundary interpolation is imposed by prescribing the boundary rows and columns of the control net, while the interior control points are selected by a Dirichlet principle: for each admissible choice of B\'ezier basis shape parameters, we compute the unique Dirichlet-energy extremal within the corresponding GT--B\'ezier patch space, which yields a parameter-dependent symmetric linear system for the interior control net under standard nondegeneracy assumptions. The remaining design freedom is thereby reduced to a four-parameter optimization problem, which we solve by particle swarm optimization. Numerical experiments show that the resulting two-level procedure consistently decreases the Dirichlet energy and, in our tests, often reduces the realized surface area relative to classical Bernstein--B\'ezier Dirichlet patches and representative quasi-harmonic and bending-energy constructions under identical boundary control data. We further adapt the same Dirichlet-extremal methodology to a hybrid tensor-product/bilinear Coons framework, obtaining minimality-biased TB--Coons patches from sparse boundary specifications.

\end{abstract}

\section{Introduction}
\label{sec:intro}

Minimal surfaces arise as stationary points of the area functional under fixed boundary data and occupy a central role in geometric analysis and its applications; see, e.g., \cite{alarcon2021minimal,monterde2004bezier2,kapfer2011minimal,burton2003minimal,colding2011course6,dierkes2010minimal}. Their ubiquity in capillarity and membrane equilibria motivates the continued study \cite{traasdahl2011high4}, while modern engineering and design, including architecture, materials, marine structures and bioinspired morphogenesis, often exploit their geometric efficiency \cite{yoo2011computer5,pan2011construction7}. At smaller scales, periodic and bicontinuous configurations connect minimal surface theory to nanostructuring and molecular engineering \cite{wang2007periodic8}.

Let $D\subset\mathbb{R}^2$ be a bounded simply connected Lipschitz domain. Given a rectifiable Jordan curve $\Gamma\subset\mathbb{R}^3$ and a continuous weakly monotone parametrization $\gamma:\partial D\to \Gamma$, consider maps $S\in W^{1,2}(D;\mathbb{R}^3)$ with trace $S|_{\partial D}=\gamma$. Writing
\[
G_S \coloneqq (\nabla S)^\top \nabla S
\qquad\text{and}\qquad
J_S \coloneqq \sqrt{\det G_S},
\]
the (parametric) Plateau problem may be expressed as
\begin{equation}
\inf\Big\{\mathcal{A}[S] : S\in W^{1,2}(D;\mathbb{R}^3),\ S|_{\partial D}=\gamma\Big\},
\qquad
\mathcal{A}[S]\coloneqq \int_D J_S\,\mathrm{d}x .
\label{eq:plat_problem}
\end{equation}
For smooth immersions, criticality of $\mathcal{A}$ is equivalent to vanishing mean curvature. Direct minimization of $\mathcal{A}$ over parametrizations is analytically delicate 
due to reparametrization invariance and degenerating minimizing sequences. 
Following Douglas' harmonic relaxation \cite{douglas1931solution}, 
modern approaches to this problem include \cite{colding2011minimal} 
for theoretical aspects and \cite{crane2013digital} for computational 
implementations.
\begin{equation}
\mathcal{D}[S]\coloneqq \tfrac12 \int_D \operatorname{tr}(G_S)\,\mathrm{d}x
= \tfrac12 \int_D \big(\|\partial_1 S\|^2+\|\partial_2 S\|^2\big)\,\mathrm{d}x.
\label{eq:Dirichlet_energy}
\end{equation}
The pointwise arithmetic--geometric mean inequality yields
\[
J_S=\sqrt{\det G_S}\le \tfrac12 \operatorname{tr}(G_S)
\quad\text{a.e. on }D,
\qquad \text{hence}\qquad \mathcal{A}[S]\le \mathcal{D}[S].
\]
Moreover, equality holds if and only if $S$ is conformal (isothermal), i.e.
\[
\langle \partial_1 S,\partial_2 S\rangle = 0
\qquad\text{and}\qquad
\|\partial_1 S\|=\|\partial_2 S\|
\quad\text{a.e. on }D,
\]
so that Dirichlet minimizers in a conformal gauge recover minimal surfaces \cite{osserman2013survey}.

In geometric modeling one replaces the infinite-dimensional admissible class by a finite-dimensional ansatz. The classical ``Plateau--B\'ezier problem'' \cite{monterde2003Plateau,hao2012minimal} restricts $S$ to tensor-product B\'ezier patches with boundary control points fixed to interpolate prescribed curves. Although the inequality $\mathcal{A}\le \mathcal{D}$ does not, in general, yield a rigorous comparison of discrete minimizers (since conformality is not enforced within a fixed discrete space), Dirichlet extremals are known to provide reliable approximations of area extremals in many practical settings \cite{monterde2003Plateau,monterde2004bezier2}. Closely related discretizations employ finite elements and variational formulations \cite{barrett2007parametric}, multiresolution B-spline schemes leading to sparse linear systems for near-minimal surfaces \cite{hao2013approximationn15}, and analyses of Dirichlet-type energies in polyhedral settings \cite{polthier2002polyhedral12}; see also modern discrete differential geometry and geometry-processing treatments
connecting conformal parameterization, Laplacians, and energy minimization
\cite{crane2018discrete,mullen2008spectral,yueh2017efficient}. Additional strands include quasi-harmonic constructions in non-polynomial spaces \cite{hao2012minimal}, low-subdivision modeling \cite{pan2011constructionn27}, explicit low-degree polynomial minimal patches \cite{xu2010quinticn31,xu2008parametricn32}, geodesically constrained approximations \cite{li2022geodesic}, and canonical factor formulations for minimal surface equations \cite{kassabov2014transitionn19}. Variational area-reduction flows within parametric classes have also been explored \cite{ahmad2014variationaln1}.

In this work we adopt the GT-B\'ezier basis \cite{ammad2022novel}, a generalized trigonometric--B\'ezier system of degree $\ge 2$ enjoying partition-of-unity, nonnegativity, and symmetry properties analogous to the Bernstein basis, while introducing two shape parameters per univariate direction. Henceforth, we specialize in the parametric domain
$D=[0,1]^2$, and encode the boundary condition by fixing boundary control rows and columns to interpolate the prescribed boundary curves. Let
\[
\big\{G_{i,m}(\,\cdot\,;(\alpha_1,\alpha_2))\big\}_{i=0}^m,
\qquad
\big\{G_{j,n}(\,\cdot\,;(\beta_1,\beta_2))\big\}_{j=0}^n
\]
denote the GT-B\'ezier basis families in the $u$- and $v$-directions, respectively. A GT-B\'ezier tensor-product patch of bidegree $(m,n)$ is
\begin{equation}
S_{m,n}(u,v;\boldsymbol{\alpha},\mathbf{P})
\;=\;
\sum_{i=0}^m \sum_{j=0}^n
P_{i,j}\, G_{i,m}\big(u;(\alpha_1,\alpha_2)\big)\, G_{j,n}\big(v;(\beta_1,\beta_2)\big),
\qquad (u,v)\in[0,1]^2,
\label{eq:GT_surface}
\end{equation}
where $\mathbf{P}=\{P_{i,j}\}_{0\le i\le m,\ 0\le j\le n}\subset\mathbb{R}^3$ is the control net. The boundary control points $\{P_{i,j}\}$ with $i\in\{0,m\}$ or $j\in\{0,n\}$ are prescribed by the boundary interpolation constraints, whereas the interior control points
\[
\mathbf{p}\;\coloneqq\;\{P_{i,j}\}_{1\le i\le m-1,\ 1\le j\le n-1}\in\mathbb{R}^{3(m-1)(n-1)}
\]
remain unknown. The shape-parameter vector is
\begin{equation}
\boldsymbol{\alpha}\coloneqq(\alpha_1,\alpha_2,\beta_1,\beta_2)\in\mathcal{P}\subset\mathbb{R}^4,
\label{eq:alpha_def}
\end{equation}
where $\mathcal{P}$ is an admissible parameter set chosen so that the GT-B\'ezier basis is well-defined and retains its shape-preserving properties. In particular, letting $\Theta\subset\mathbb{R}^2$ denote the admissible univariate parameter set for the GT-B\'ezier basis (cf.\ Section~\ref{sec:def}), we assume
\begin{equation}
\mathcal{P}\subseteq \Theta\times \Theta .
\label{eq:P_subset_Theta}
\end{equation}
For each fixed $\boldsymbol{\alpha}\in\mathcal{P}$, consider the discrete admissible class
\[
\mathcal{S}_{m,n}(\boldsymbol{\alpha})
\;\coloneqq\;
\Big\{S_{m,n}(\cdot,\cdot;\boldsymbol{\alpha},\mathbf{P}) : \mathbf{P}\ \text{satisfies the boundary interpolation constraints}\Big\}.
\]
We define the discrete Dirichlet extremal (indeed, minimizer in the typical strictly convex case) by
\begin{equation}
S^\star(\boldsymbol{\alpha})
\in \arg\min\big\{\mathcal{D}[S] : S\in \mathcal{S}_{m,n}(\boldsymbol{\alpha})\big\}.
\label{eq:discrete_dirichlet_extremal}
\end{equation}
Taking variations supported on interior degrees of freedom yields the normal equations: for each interior index pair $(k,\ell)\in\{1,\dots,m-1\}\times\{1,\dots,n-1\}$ and each component $a\in\{1,2,3\}$, let
\[
\Phi_{k,\ell,a}(u,v)
\;\coloneqq\;
G_{k,m}\big(u;(\alpha_1,\alpha_2)\big)\,G_{\ell,n}\big(v;(\beta_1,\beta_2)\big)\,e_a,
\]
so that $\Phi_{k,\ell,a}|_{\partial([0,1]^2)}=0$. Then the first variation of $\mathcal{D}$ at $S^\star(\boldsymbol{\alpha})$ yields
\begin{equation}
\int_{[0,1]^2} \big\langle \nabla S^\star(\boldsymbol{\alpha}), \nabla \Phi_{k,\ell,a}\big\rangle\,\mathrm{d}u\,\mathrm{d}v \;=\; 0.
\label{eq:normal_equations}
\end{equation}
Equivalently, the interior control vector $\mathbf{p}=\mathbf{p}(\boldsymbol{\alpha})$ satisfies a parameter-dependent linear system
\begin{equation}
A(\boldsymbol{\alpha})\,\mathbf{p}(\boldsymbol{\alpha}) \;=\; \mathbf{b}(\boldsymbol{\alpha}),
\label{eq:param_linear_system}
\end{equation}
where $A(\boldsymbol{\alpha})$ and $\mathbf{b}(\boldsymbol{\alpha})$ are determined explicitly by the GT-B\'ezier basis functions and their derivatives and by the fixed boundary control data. On the admissible set $\mathcal{P}$ we assume $A(\boldsymbol{\alpha})$ is nonsingular, so that
\[
\mathbf{p}(\boldsymbol{\alpha}) \;=\; A(\boldsymbol{\alpha})^{-1}\mathbf{b}(\boldsymbol{\alpha})
\]
is well-defined.

We then introduce the reduced Dirichlet functional
\begin{equation}
\mathcal{J}:\mathcal{P}\to\mathbb{R},
\qquad
\mathcal{J}(\boldsymbol{\alpha})
\;\coloneqq\;
\mathcal{D}\big[S^\star(\boldsymbol{\alpha})\big]
\;=\;
\mathcal{D}\!\Big[S_{m,n}\big(\cdot,\cdot;\boldsymbol{\alpha},\mathbf{P}(\boldsymbol{\alpha})\big)\Big],
\label{eq:Dirichlet_functional}
\end{equation}
where $\mathbf{P}(\boldsymbol{\alpha})$ denotes the full control net obtained by combining the fixed boundary control points with the interior solution $\mathbf{p}(\boldsymbol{\alpha})$. We minimize $\mathcal{J}$ over $\mathcal{P}$ using particle swarm optimization (PSO), a widely used derivative-free metaheuristic; see, e.g.,
\cite{gad2022pso,shami2022particlesur}.
\[
\boldsymbol{\alpha}^\star \in \arg\min_{\boldsymbol{\alpha}\in\mathcal{P}} \mathcal{J}(\boldsymbol{\alpha}).
\]
In each evaluation of $\mathcal{J}$, the dependence on $\boldsymbol{\alpha}$ enters solely through the GT-B\'ezier basis (and its derivatives) and the solution of \eqref{eq:param_linear_system}; no \emph{a priori} numerical specialization of $\boldsymbol{\alpha}$ is required at the level of the variational formulation.

This construction retains the established strengths of Dirichlet extremals for Plateau--B\'ezier-type problems \cite{hao2020quasi,monterde2004bezier2,osserman2013survey} while using the additional geometric degrees of freedom afforded by the GT-B\'ezier basis \cite{ammad2022novel}. In practice, these shape parameters often permit lower Dirichlet energies than classical polynomial tensor-product B\'ezier patches under identical boundary data; empirically, this frequently correlates with reduced realized surface area. We further investigate the harmonic subclass of GT-B\'ezier patches, its dependence on boundary rows and columns, and its relationship to minimality in free-boundary regimes, connecting with harmonic parametrizations and discrete minimal surface literature \cite{lam2018discrete,ahmad2022computational,polthier2002polyhedral12,hao2012minimal}. As an application, we employ the same Dirichlet-extremal principle to infer interior control points for TB--Coons blends, producing minimality-biased GT--Coons surfaces from parsimonious control specifications.

Section~\ref{sec:def} recalls definitions and basic properties of classical and GT-B\'ezier surfaces. Section~\ref{sec:dirichlet-extremals} formulates the Dirichlet-extremal construction for GT-B\'ezier patches and the associated reduced functional. Section~\ref{sec:pso} presents the particle swarm optimization procedure used to minimize $\mathcal{J}$, together with an explicit algorithm. Section~\ref{sec:examples} reports numerical examples for the Plateau--GT-B\'ezier problem. Section~\ref{sec:harmonic} discusses the relationship between minimal surfaces and harmonic surfaces in our framework. Section~\ref{sec:coon} applies the Dirichlet-extremal method in the trigonometric basis to construct minimal TB--Coons surfaces. Finally, Section~\ref{sec:conclusion} summarizes conclusions and future directions.

\section{Preliminaries: GT-B\'ezier basis functions and shape-parameter effects}
\label{sec:def}

We recall the classical Bernstein basis and the generalized trigonometric (GT)-B\'ezier basis of Ammad et al.~\cite{ammad2022novel}, adopting notation consistent with Section~\ref{sec:intro}. Unless stated otherwise, the univariate parameter is $t\in[0,1]$. For the GT-B\'ezier basis we use a \emph{univariate} shape-parameter pair
\[
\boldsymbol{\theta}\coloneqq(\theta_1,\theta_2)\in \Theta \subset \mathbb{R}^2,
\qquad
\Theta \coloneqq [0.5,3.5]^2,
\]
thereby avoiding collision with the \emph{surface} shape vector
$\boldsymbol{\alpha}=(\alpha_1,\alpha_2,\beta_1,\beta_2)\in\mathcal{P}\subset\mathbb{R}^4$
introduced in Section~\ref{sec:intro}. In view of \eqref{eq:P_subset_Theta}, each tensor-product patch uses the two instantiated pairs
\[
\boldsymbol{\alpha}^{(u)}=(\alpha_1,\alpha_2)\in\Theta,
\qquad
\boldsymbol{\alpha}^{(v)}=(\beta_1,\beta_2)\in\Theta.
\]

\subsection{Bernstein polynomials and B\'ezier representation}

\begin{definition}[Bernstein polynomials {\cite[Ch.~1]{farouki2012bernstein,Phillips2003}}]\label{def:bernstein}
Let $d\in\mathbb{N}_0$. For each $i=0,\dots,d$, the $i$th Bernstein polynomial of degree $d$
on $[0,1]$ is defined by
\begin{equation}\label{eq:bernstein}
B_{i,d}(t):=\binom{d}{i} t^i(1-t)^{d-i}, \qquad t\in[0,1].
\end{equation}
The collection $\{B_{i,d}\}_{i=0}^d$ is called the Bernstein basis of degree $d$ on $[0,1]$.
\end{definition}

\begin{definition}[B\'ezier curve {\cite{farin2014curves}}]\label{def:bezier}
Let $c_0,\dots,c_d\in\mathbb{R}^m$ (typically $m\in\{2,3\}$) be control points.
The (parametric) B\'ezier curve of degree $d$ associated with $\{c_i\}_{i=0}^d$ is the map
$C:[0,1]\to\mathbb{R}^m$ given by
\begin{equation}\label{eq:bezier}
C(t):=\sum_{i=0}^d c_i\,B_{i,d}(t), \qquad t\in[0,1],
\end{equation}
where $B_{i,d}$ are the Bernstein polynomials from Definition~\ref{def:bernstein}.
\end{definition}

\subsection{GT-B\'ezier basis: base case and degree elevation}

\begin{definition}[Quadratic GT-B\'ezier (base case)]
\label{def:gt_seed}
Fix $\boldsymbol{\theta}=(\theta_1,\theta_2)\in\Theta$. The quadratic GT-B\'ezier basis on $[0,1]$ is defined by
\begin{equation}
\label{eq:gt2}
\begin{aligned}
G_{0,2}(t;\boldsymbol{\theta})
&= \frac{\theta_1}{2}\Big(\sin^2\!\big(\tfrac{\pi}{2}t\big)-\sin\!\big(\tfrac{\pi}{2}t\big)\Big)
    + \cos^2\!\big(\tfrac{\pi}{2}t\big),\\
G_{2,2}(t;\boldsymbol{\theta})
&= \frac{\theta_2}{2}\Big(\cos^2\!\big(\tfrac{\pi}{2}t\big)-\cos\!\big(\tfrac{\pi}{2}t\big)\Big)
    + \sin^2\!\big(\tfrac{\pi}{2}t\big),\\
G_{1,2}(t;\boldsymbol{\theta})
&= 1 - G_{0,2}(t;\boldsymbol{\theta}) - G_{2,2}(t;\boldsymbol{\theta}).
\end{aligned}
\end{equation}
\end{definition}

\begin{definition}[Recursive degree elevation (GT-B\'ezier basis of degree $n\ge 3$)]
\label{def:gt_rec}
For $n\ge 3$ and $k=0,\dots,n$, define
\begin{equation}
\label{eq:generalbasis}
G_{k,n}(t;\boldsymbol{\theta})
\;=\;
(1-t)\,G_{k,n-1}(t;\boldsymbol{\theta}) + t\,G_{k-1,n-1}(t;\boldsymbol{\theta}),
\end{equation}
with the convention $G_{k,n}\equiv 0$ whenever $k<0$ or $k>n$. The family
$\{G_{k,n}(\cdot;\boldsymbol{\theta})\}_{k=0}^n$
is called the GT-B\'ezier basis of degree $n$ associated with the seed
\eqref{eq:gt2}; cf.~\cite{ammad2022novel}.
\end{definition}

\begin{proposition}[Partition of unity and endpoint interpolation]
\label{prop:gt_pu_endpoints}
For each $n\ge 2$ and $\boldsymbol{\theta}\in\Theta$,
\[
\sum_{k=0}^n G_{k,n}(t;\boldsymbol{\theta}) = 1 \quad \text{for all } t\in[0,1],
\]
and
\[
G_{0,n}(0;\boldsymbol{\theta})=1,\qquad G_{n,n}(1;\boldsymbol{\theta})=1,\qquad
G_{k,n}(0;\boldsymbol{\theta})=G_{k,n}(1;\boldsymbol{\theta})=0\ \ (k\notin\{0,n\}).
\]
\end{proposition}

\begin{proof}
The endpoint identities follow from the explicit seed \eqref{eq:gt2} and the recursion \eqref{eq:generalbasis} by induction on $n$.
The partition-of-unity property is verified for $n=2$ from \eqref{eq:gt2} and is propagated to $n\ge 3$ by summing \eqref{eq:generalbasis} over $k$ and using the induction hypothesis.
\end{proof}

\begin{remark}[Link to the surface parameters of Section~\ref{sec:intro}]
\label{rem:link_surface_params}
In the bivariate setting \eqref{eq:GT_surface}, we employ two independent GT families:
\[
\{G_{i,m}(\cdot;\boldsymbol{\alpha}^{(u)})\}_{i=0}^m \ \text{in the $u$-direction},
\qquad
\{G_{j,n}(\cdot;\boldsymbol{\alpha}^{(v)})\}_{j=0}^n \ \text{in the $v$-direction}.
\]
Thus the univariate shape-pair $\boldsymbol{\theta}$ used in this section is instantiated as
$\boldsymbol{\theta}=\boldsymbol{\alpha}^{(u)}$ (respectively $\boldsymbol{\alpha}^{(v)}$) when constructing the tensor-product patch.
\end{remark}

\subsection{Influence of Shape Parameters on GT-B\'ezier Curves}

\begin{definition}[Univariate GT-B\'ezier curve]
\label{def:gt-curve}
Given control points $\{b_k\}_{k=0}^n\subset\mathbb{R}^m$ ($m\in\{2,3\}$), degree $n\ge 2$, and the GT-B\'ezier basis
$\{G_{k,n}(\cdot;\boldsymbol{\theta})\}_{k=0}^n$,
the associated GT-B\'ezier curve is
\begin{equation}
\label{eq:gt-curve}
F(t;\boldsymbol{\theta}) \;=\; \sum_{k=0}^n G_{k,n}(t;\boldsymbol{\theta})\,b_k,
\qquad t\in[0,1].
\end{equation}
\end{definition}

\noindent\textbf{Curvature diagnostics.}
Curvature plots are a standard diagnostic for shape regularity in geometric design. Restricting to planar curves $F(\cdot;\boldsymbol{\theta})\subset\mathbb{R}^2$, suppose $F\in C^2([0,1];\mathbb{R}^2)$ is regular, i.e.\ $F'(t)\neq 0$ for all $t$. Its signed curvature is
\begin{equation}
\label{eq:curvature}
\kappa_{F(\cdot;\boldsymbol{\theta})}(t)
\;=\;
\frac{\det\!\big(F'(t;\boldsymbol{\theta}),\,F''(t;\boldsymbol{\theta})\big)}{\|F'(t;\boldsymbol{\theta})\|^3},
\qquad t\in[0,1].
\end{equation}
In particular, the dependence of the curvature on the shape parameters may be regarded as the nonlinear operator
\[
\mathscr{K}:\Theta\to L^\infty(0,1),
\qquad
\mathscr{K}(\boldsymbol{\theta}) \coloneqq \kappa_{F(\cdot;\boldsymbol{\theta})}(\cdot),
\]
whose nontriviality stems from the trigonometric seed \eqref{eq:gt2} and its propagation through the recursion \eqref{eq:generalbasis}. Empirically, varying $\boldsymbol{\theta}$ redistributes parametric stretch along $t$, thereby shifting the location and magnitude of curvature extrema while preserving endpoint interpolation (Proposition~\ref{prop:gt_pu_endpoints}) and the affine-combination structure of \eqref{eq:gt-curve}. When, additionally, nonnegativity holds for the chosen admissible set $\Theta$ (as in \cite{ammad2022novel}), the curve remains within the convex hull of its control polygon.

Figure~\ref{fig:ckc} illustrates representative cubic GT-B\'ezier curves under different parameter selections alongside their curvature plots, demonstrating modulation of curvature magnitude and its spatial allocation induced by the shape variables.

\begin{figure}[t]
\centering
\subfigure[Classical B\'ezier curve]{\includegraphics[width=1in]{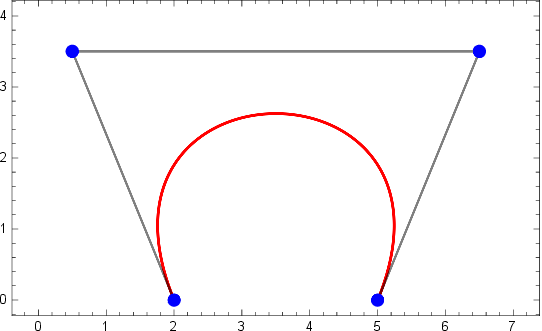}}
\subfigure[GT, $\theta_1=\theta_2=0.5$]{\includegraphics[width=1in]{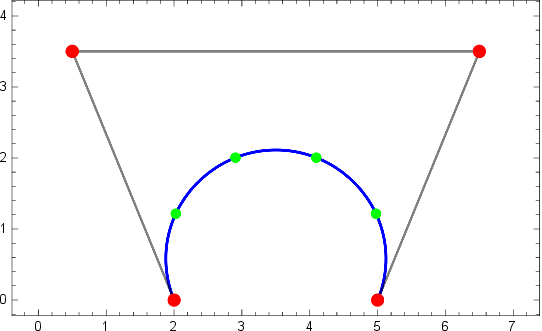}}
\subfigure[GT, $\theta_1=\theta_2=3.5$]{\includegraphics[width=1in]{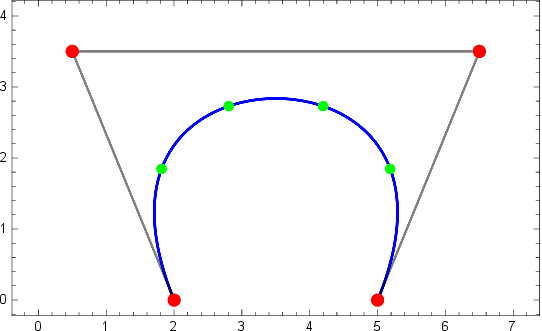}}
\subfigure[GT, $\theta_1=3.5,\ \theta_2=0.5$]{\includegraphics[width=1in]{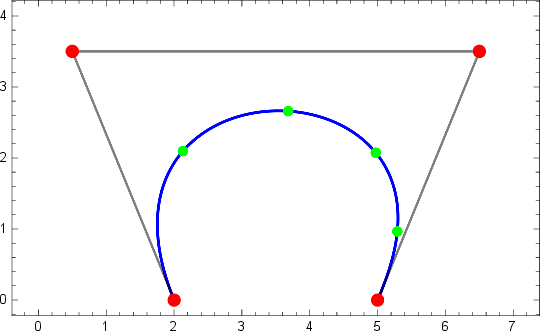}}
\subfigure[GT, $\theta_1=0.5,\ \theta_2=3.5$]{\includegraphics[width=1in]{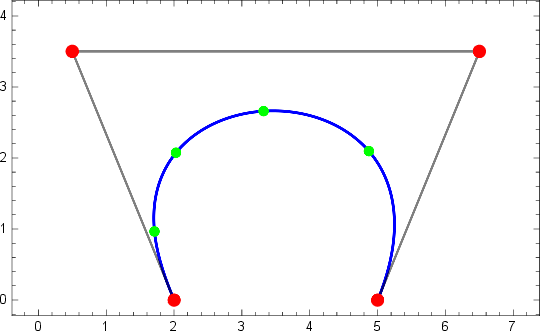}}\\
\subfigure[Curvature plot]{\includegraphics[width=1in]{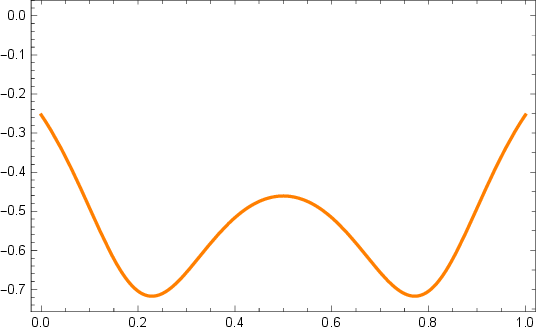}}
\subfigure[Curvature plot]{\includegraphics[width=1in]{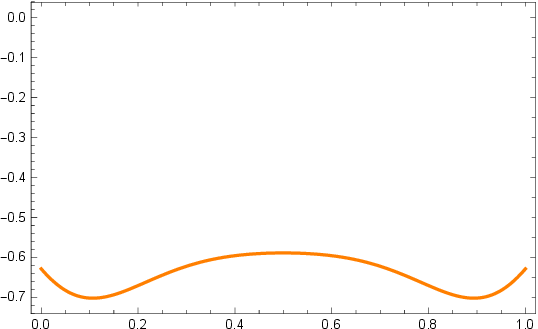}}
\subfigure[Curvature plot]{\includegraphics[width=1in]{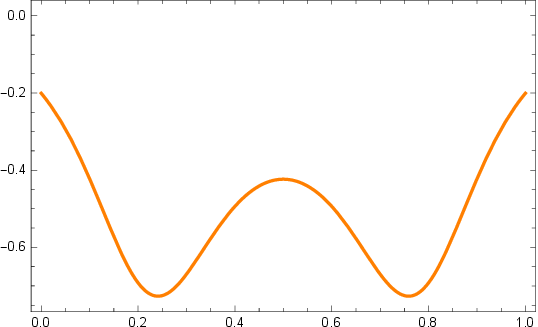}}
\subfigure[Curvature plot]{\includegraphics[width=1in]{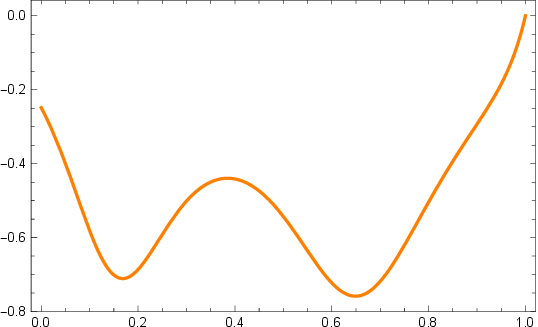}}
\subfigure[Curvature plot]{\includegraphics[width=1in]{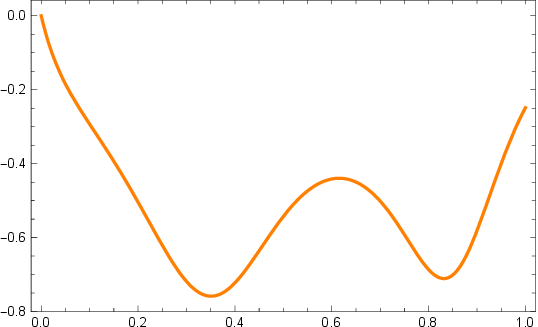}}
\caption{Cubic GT-B\'ezier curves and corresponding curvature plots under varying shape parameters.}
\label{fig:ckc}
\end{figure}

\section{Extremals of the Dirichlet functional for GT-B\'ezier surfaces}
\label{sec:dirichlet-extremals}

We work on $D=[0,1]^2$ and consider GT-B\'ezier  surfaces
$S_{m,n}(\cdot,\cdot;\boldsymbol{\alpha},\mathbf{P})$ with fixed boundary control rows and columns. In this section we express the condition that a GT-B\'ezier surface is an extremal of the Dirichlet functional (with prescribed boundary) as an explicit linear system in the interior control points. Rather than deriving the Euler--Lagrange equation in a strong form, we view the Dirichlet energy as a real-valued function on the Euclidean space of interior control variables and characterize its critical points by vanishing of the gradient. Closed-form recursive formulas for the derivatives of the GT-B\'ezier basis, used to evaluate the integrals appearing below, are provided in Appendix~\ref{app:derivative-recursion}.

\subsection{Tensor-product GT-B\'ezier surface and Dirichlet functional}

Let $m,n\ge 2$ and let $\boldsymbol{\alpha}=(\alpha_1,\alpha_2,\beta_1,\beta_2)\in\mathcal{P}$ be the surface shape vector. For notational convenience we set
\begin{equation}
\boldsymbol{\alpha}^{(u)}\coloneqq(\alpha_1,\alpha_2)\in\Theta,
\qquad
\boldsymbol{\alpha}^{(v)}\coloneqq(\beta_1,\beta_2)\in\Theta.
\label{eq:alpha_uv}
\end{equation}
We denote the corresponding GT-B\'ezier basis by $G_{i,m}(u;\boldsymbol{\alpha}^{(u)})$ and
$G_{j,n}(v;\boldsymbol{\alpha}^{(v)})$.
Given control points  $\{P_{ij}\}_{i,j=0}^{m,n}\subset\mathbb{R}^3$, define the GT-B\'ezier tensor-product surface
\begin{equation}
\label{eq:GT_surface_dirichlet}
S(u,v;\boldsymbol{\alpha},\mathbf{P})
\;=\;
\sum_{i=0}^{m}\sum_{j=0}^{n} P_{ij}\,G_{i,m}(u;\boldsymbol{\alpha}^{(u)})\,G_{j,n}(v;\boldsymbol{\alpha}^{(v)}),
\qquad (u,v)\in[0,1]^2 .
\end{equation}
Its Dirichlet functional is
\begin{equation}
\label{eq:dirichlet_energy}
\mathcal{D}[S]
\;\coloneqq\;
\frac12\int_{0}^{1}\!\!\int_{0}^{1}
\Big(\|S_u(u,v)\|^2+\|S_v(u,v)\|^2\Big)\,du\,dv .
\end{equation}

\subsection{Extremal condition as a linear system in the control points}

\begin{theorem}[Dirichlet extremals for GT-B\'ezier surfaces]
\label{thm:df}
Let $S$ be the GT-B\'ezier surface \eqref{eq:GT_surface_dirichlet} with prescribed boundary control points
$\{P_{ij}:\ i\in\{0,m\}\ \text{or}\ j\in\{0,n\}\}$.
Then $S$ is an extremal of the Dirichlet functional \eqref{eq:dirichlet_energy}
(with respect to variations of the interior control points) if and only if, for every
$k\in\{1,\dots,m-1\}$, $\ell\in\{1,\dots,n-1\}$, and each $\mu\in\{1,2,3\}$,
\begin{multline}
\label{eq:df_system}
\sum_{i=0}^{m-1}\sum_{j=0}^{n}
I_{k,i}^{(1)}\,J_{\ell,j}^{(1)}\,\big\langle e^\mu,\;P_{i+1,j}-P_{i,j}\big\rangle
\;+\;
\sum_{i=0}^{m-1}\sum_{j=0}^{n}
I_{k,i}^{(2)}\,J_{\ell,j}^{(1)}\,\big\langle e^\mu,\;P_{i,j}\big\rangle
\\
\;+\;
\sum_{i=0}^{m}\sum_{j=0}^{n-1}
I_{k,i}^{(3)}\,J_{\ell,j}^{(2)}\,\big\langle e^\mu,\;P_{i,j+1}-P_{i,j}\big\rangle
\;+\;
\sum_{i=0}^{m}\sum_{j=0}^{n-1}
I_{k,i}^{(3)}\,J_{\ell,j}^{(3)}\,\big\langle e^\mu,\;P_{i,j}\big\rangle
\;=\;0.
\end{multline}
Here $e^\mu$ denotes the $\mu$-th canonical basis vector of $\mathbb{R}^3$ and
$\langle\cdot,\cdot\rangle$ is the Euclidean inner product. The scalar coefficients are
\begin{align}
\label{eq:I_defs}
I_{k,i}^{(1)}
&\coloneqq \int_0^1 G'_{k,m}(u;\boldsymbol{\alpha}^{(u)})
\Big( G_{i,m-1}(u;\boldsymbol{\alpha}^{(u)}) + u\, G'_{i,m-1}(u;\boldsymbol{\alpha}^{(u)}) \Big)\,du,
\\
I_{k,i}^{(2)}
&\coloneqq \int_0^1 G'_{k,m}(u;\boldsymbol{\alpha}^{(u)})\,G'_{i,m-1}(u;\boldsymbol{\alpha}^{(u)})\,du,
\\
I_{k,i}^{(3)}
&\coloneqq \int_0^1 G_{k,m}(u;\boldsymbol{\alpha}^{(u)})\,G_{i,m}(u;\boldsymbol{\alpha}^{(u)})\,du,
\\[2mm]
\label{eq:J_defs}
J_{\ell,j}^{(1)}
&\coloneqq \int_0^1 G_{\ell,n}(v;\boldsymbol{\alpha}^{(v)})\,G_{j,n}(v;\boldsymbol{\alpha}^{(v)})\,dv,
\\
J_{\ell,j}^{(2)}
&\coloneqq \int_0^1 G'_{\ell,n}(v;\boldsymbol{\alpha}^{(v)})
\Big( G_{j,n-1}(v;\boldsymbol{\alpha}^{(v)}) + v\, G'_{j,n-1}(v;\boldsymbol{\alpha}^{(v)}) \Big)\,dv,
\\
J_{\ell,j}^{(3)}
&\coloneqq \int_0^1 G'_{\ell,n}(v;\boldsymbol{\alpha}^{(v)})\,G'_{j,n-1}(v;\boldsymbol{\alpha}^{(v)})\,dv.
\end{align}
\end{theorem}

\begin{proof}
Write each control point as $P_{ij}=(x_{ij}^1,x_{ij}^2,x_{ij}^3)\in\mathbb{R}^3$.
For fixed $(k,\ell)$ with $1\le k\le m-1$, $1\le \ell\le n-1$ and $\mu\in\{1,2,3\}$, consider the partial derivative
$\partial \mathcal{D}[S]/\partial x_{k\ell}^\mu$.
Using the standard identity (e.g., \cite{monterde2004bezier2})
\[
\frac{\partial \mathcal{D}[S]}{\partial x_{k\ell}^\mu}
=
\int_{0}^{1}\!\!\int_{0}^{1}
\left(
\left\langle S_u,\frac{\partial S_u}{\partial x_{k\ell}^\mu}\right\rangle
+
\left\langle S_v,\frac{\partial S_v}{\partial x_{k\ell}^\mu}\right\rangle
\right)\,du\,dv,
\]
it remains to compute $S_u$, $S_v$, and the variations $\partial S_u/\partial x_{k\ell}^\mu$,
$\partial S_v/\partial x_{k\ell}^\mu$.

From \eqref{eq:GT_surface_dirichlet},
\[
\frac{\partial S}{\partial x_{k\ell}^\mu}(u,v)
=
G_{k,m}(u;\boldsymbol{\alpha}^{(u)})\,G_{\ell,n}(v;\boldsymbol{\alpha}^{(v)})\,e^\mu,
\]
hence by commuting partial derivatives,
\[
\frac{\partial S_u}{\partial x_{k\ell}^\mu}(u,v)
=
G'_{k,m}(u;\boldsymbol{\alpha}^{(u)})\,G_{\ell,n}(v;\boldsymbol{\alpha}^{(v)})\,e^\mu,
\qquad
\frac{\partial S_v}{\partial x_{k\ell}^\mu}(u,v)
=
G_{k,m}(u;\boldsymbol{\alpha}^{(u)})\,G'_{\ell,n}(v;\boldsymbol{\alpha}^{(v)})\,e^\mu.
\]

Next, we use the derivative representation of the GT-B\'ezier surface (cf.\ the formulas collected in
Appendix~\ref{app:derivative-recursion}) to write
\begin{align*}
S_u(u,v)
&=
\sum_{i=0}^{m-1}\sum_{j=0}^{n}
\Big( G_{i,m-1}(u;\boldsymbol{\alpha}^{(u)}) + u\,G'_{i,m-1}(u;\boldsymbol{\alpha}^{(u)}) \Big)
G_{j,n}(v;\boldsymbol{\alpha}^{(v)})\,(P_{i+1,j}-P_{i,j})
\\
&\quad+
\sum_{i=0}^{m-1}\sum_{j=0}^{n}
G'_{i,m-1}(u;\boldsymbol{\alpha}^{(u)})\,G_{j,n}(v;\boldsymbol{\alpha}^{(v)})\,P_{i,j},
\\[1mm]
S_v(u,v)
&=
\sum_{i=0}^{m}\sum_{j=0}^{n-1}
G_{i,m}(u;\boldsymbol{\alpha}^{(u)})
\Big( G_{j,n-1}(v;\boldsymbol{\alpha}^{(v)}) + v\,G'_{j,n-1}(v;\boldsymbol{\alpha}^{(v)}) \Big)\,(P_{i,j+1}-P_{i,j})
\\
&\quad+
\sum_{i=0}^{m}\sum_{j=0}^{n-1}
G_{i,m}(u;\boldsymbol{\alpha}^{(u)})\,G'_{j,n-1}(v;\boldsymbol{\alpha}^{(v)})\,P_{i,j}.
\end{align*}

Substituting these expressions into the gradient formula and using bilinearity of $\langle\cdot,\cdot\rangle$ gives
\begin{align*}
\frac{\partial \mathcal{D}[S]}{\partial x_{k\ell}^\mu}
&=
\sum_{i=0}^{m-1}\sum_{j=0}^{n}
\big\langle e^\mu,P_{i+1,j}-P_{i,j}\big\rangle
\int_0^1\!\!\int_0^1
G'_{k,m}(u;\boldsymbol{\alpha}^{(u)})
\Big( G_{i,m-1}(u;\boldsymbol{\alpha}^{(u)}) + u\,G'_{i,m-1}(u;\boldsymbol{\alpha}^{(u)}) \Big)
\\[-1mm]
&\hspace{44mm}\times
G_{\ell,n}(v;\boldsymbol{\alpha}^{(v)})\,G_{j,n}(v;\boldsymbol{\alpha}^{(v)})\,du\,dv
\\
&\quad+
\sum_{i=0}^{m-1}\sum_{j=0}^{n}
\big\langle e^\mu,P_{i,j}\big\rangle
\int_0^1\!\!\int_0^1
G'_{k,m}(u;\boldsymbol{\alpha}^{(u)})\,G'_{i,m-1}(u;\boldsymbol{\alpha}^{(u)})
\,G_{\ell,n}(v;\boldsymbol{\alpha}^{(v)})\,G_{j,n}(v;\boldsymbol{\alpha}^{(v)})\,du\,dv
\\
&\quad+
\sum_{i=0}^{m}\sum_{j=0}^{n-1}
\big\langle e^\mu,P_{i,j+1}-P_{i,j}\big\rangle
\int_0^1\!\!\int_0^1
G_{k,m}(u;\boldsymbol{\alpha}^{(u)})\,G_{i,m}(u;\boldsymbol{\alpha}^{(u)})
\\[-1mm]
&\hspace{44mm}\times
G'_{\ell,n}(v;\boldsymbol{\alpha}^{(v)})
\Big( G_{j,n-1}(v;\boldsymbol{\alpha}^{(v)}) + v\,G'_{j,n-1}(v;\boldsymbol{\alpha}^{(v)}) \Big)\,du\,dv
\\
&\quad+
\sum_{i=0}^{m}\sum_{j=0}^{n-1}
\big\langle e^\mu,P_{i,j}\big\rangle
\int_0^1\!\!\int_0^1
G_{k,m}(u;\boldsymbol{\alpha}^{(u)})\,G_{i,m}(u;\boldsymbol{\alpha}^{(u)})\,
G'_{\ell,n}(v;\boldsymbol{\alpha}^{(v)})\,G'_{j,n-1}(v;\boldsymbol{\alpha}^{(v)})\,du\,dv.
\end{align*}
Finally, each double integral separates into a product of a $u$-integral and a $v$-integral, yielding precisely the coefficients
\eqref{eq:I_defs}--\eqref{eq:J_defs}. Therefore
$\partial \mathcal{D}[S]/\partial x_{k\ell}^\mu=0$ for all interior indices $(k,\ell)$ and components $\mu$
is equivalent to \eqref{eq:df_system}, completing the proof.
\end{proof}

Equation \eqref{eq:df_system} is the coordinate form of the weak normal equations
\eqref{eq:normal_equations} from Section~\ref{sec:intro} restricted to the GT-B\'ezier tensor-product space.
Upon ordering the interior control variables into the vector $\mathbf p\in\mathbb R^{3(m-1)(n-1)}$,
the collection of equations \eqref{eq:df_system} assembles into the linear system
$A(\boldsymbol{\alpha})\mathbf p=\mathbf b(\boldsymbol{\alpha})$ in \eqref{eq:param_linear_system},
where the entries of $A(\boldsymbol{\alpha})$ and $\mathbf b(\boldsymbol{\alpha})$ are determined by the coefficients
\eqref{eq:I_defs}--\eqref{eq:J_defs} and by the prescribed boundary control data.

By Theorem~\ref{thm:df}, prescribing the boundary control points and taking the interior control points
$\{P_{k\ell}\}_{k,\ell}^{m-1, n-1}$ as unknowns yields a square linear system
\eqref{eq:df_system}. Whenever the associated coefficient matrix is nonsingular (equivalently, the stiffness matrix is positive definite on the interior degrees of freedom), this system has a unique solution, hence a unique Dirichlet extremal in the chosen GT-B\'ezier tensor-product space.

We now connect this Dirichlet-based discretization with the geometric goal of minimizing area.
For smooth charts, the Dirichlet energy dominates the area functional and agrees with it exactly for
conformal (in particular, isothermal) parametrizations. The next theorem shows that, assuming the
existence of an isothermal area-minimizing chart and sufficiently dense boundary sampling, the areas of the resulting GT--Dirichlet extremals admit an asymptotic upper bound in terms of the minimal area.

\begin{theorem}[Asymptotic upper bound for the area of GT--Dirichlet extremals]
\label{thm:area_convergence}
Let $S:[0,1]^2\to\mathbb R^3$ be a $C^2$ immersion which is an \emph{isothermal} parametrization.
For each $d\in\mathbb N$, let $Y_d$ be the GT--B\'ezier surface of degree $(d,d)$
(with fixed shape parameters $\boldsymbol\alpha\in\mathcal P$)
which minimizes the Dirichlet functional among all GT--B\'ezier surfaces whose boundary control
points coincide with the sampled boundary net
\[
P_d \;=\;\Bigl\{S\!\left(\frac{i}{d},\frac{j}{d}\right)\Bigr\}_{i,j=0}^d
\qquad\text{on }\partial\{0,\dots,d\}^2.
\]
Assume moreover that there exists a sequence of GT--B\'ezier surfaces $X_d$ of degree $(d,d)$ with the same
boundary data $P_d$ such that $X_d\to S$ in $C^1([0,1]^2)$ as $d\to\infty$.
Then
\[
\lim_{d\to\infty}\mathcal A[X_d]=\mathcal A[S]
\qquad\text{and}\qquad
\limsup_{d\to\infty}\mathcal A[Y_d]\le \mathcal A[S],
\]
where $\mathcal A[\cdot]$ denotes the area functional.
\end{theorem}

\begin{proof}
We use three standard ingredients.

\medskip
\noindent\textbf{(i) Area is bounded by Dirichlet energy.}
For any $C^1$ map $F:[0,1]^2\to\mathbb R^3$, the arithmetic--geometric mean inequality yields
\[
\|F_u\|^2+\|F_v\|^2 \;\ge\; 2\|F_u\|\,\|F_v\|
\;\ge\; 2\|F_u\times F_v\|.
\]
Integrating gives
\begin{equation}
\label{eq:area_le_dirichlet}
\mathcal A[F]=\int_{[0,1]^2}\|F_u\times F_v\|
\;\le\;
\frac12\int_{[0,1]^2}(\|F_u\|^2+\|F_v\|^2)=\mathcal D[F].
\end{equation}

\medskip
\noindent\textbf{(ii) Equality for isothermal parametrizations.}
Since $S$ is isothermal, we have $\langle S_u,S_v\rangle=0$ and $\|S_u\|=\|S_v\|$ on $[0,1]^2$. Hence
$\|S_u\times S_v\|=\|S_u\|\,\|S_v\|=\|S_u\|^2
=\frac12(\|S_u\|^2+\|S_v\|^2)$, and therefore
\begin{equation}
\label{eq:area_eq_dirichlet_for_S}
\mathcal A[S]=\mathcal D[S].
\end{equation}

\medskip
\noindent\textbf{(iii) Minimality of $Y_d$ for the Dirichlet energy under the discrete boundary constraint.}
By construction, $Y_d$ minimizes $\mathcal D[\cdot]$ over the affine subspace
of GT charts with boundary control points fixed by $P_d$. In particular, for any admissible competitor
$F$ in the same GT--B\'ezier space with the same boundary data,
\begin{equation}
\label{eq:dir_min_property}
\mathcal D[Y_d]\le \mathcal D[F].
\end{equation}

\medskip
Now fix $d$ and take as competitor $F=X_d$, which is admissible by hypothesis (same boundary net $P_d$). Combining
\eqref{eq:area_le_dirichlet} and \eqref{eq:dir_min_property} yields
\[
\mathcal A[Y_d]\le \mathcal D[Y_d]\le \mathcal D[X_d].
\]
Since $X_d\to S$ in $C^1([0,1]^2)$, we have $X_{d,u}\to S_u$ and $X_{d,v}\to S_v$ uniformly; hence
$\mathcal D[X_d]\to \mathcal D[S]$. Using \eqref{eq:area_eq_dirichlet_for_S}, we conclude
\[
\limsup_{d\to\infty}\mathcal A[Y_d]\le \lim_{d\to\infty}\mathcal D[X_d]=\mathcal D[S]=\mathcal A[S].
\]
Finally, $\mathcal A[X_d]\to \mathcal A[S]$ follows directly from the $C^1$ convergence and continuity of the
area integrand $\|F_u\times F_v\|$ under uniform convergence of $F_u,F_v$.
\end{proof}
\begin{remark}
The conclusion is stated as a $\limsup$ because the discrete boundary constraint fixes only a sampled boundary net,
so $Y_d$ need not span the same Jordan curve $\Gamma$ for each $d$. If one enforces exact boundary interpolation
$Y_d|_{\partial D}=S|_{\partial D}$ (or assumes an appropriate stability result under boundary perturbations),
the $\limsup$ can be strengthened to a full limit.
\end{remark}
\section{Optimization of the reduced Dirichlet functional via Particle Swarm Optimization}
\label{sec:pso}

We consider the reduced Dirichlet functional
$\mathcal{J}(\boldsymbol{\alpha})=\mathcal{D}[S^\star(\boldsymbol{\alpha})]$,
where $S^\star(\boldsymbol{\alpha})$ is the GT--Dirichlet extremal obtained by solving the
interior-control linear system \eqref{eq:param_linear_system} for fixed
$\boldsymbol{\alpha}=(\alpha_1,\alpha_2,\beta_1,\beta_2)\in\mathcal{P}$. In this section we describe the numerical strategy used to minimize $\mathcal{J}$ over
the admissible parameter set $\mathcal{P}$, namely particle swarm optimization (PSO)
\cite{nayak202325,wang2018pso}.

Let $\mathcal{P}\subseteq\Theta\times\Theta$ be the admissible set for the GT-B\'ezier basis shape parameters
(cf.\ \eqref{eq:alpha_def}--\eqref{eq:P_subset_Theta}). The optimization problem was solved in the numerical experiments are
\begin{equation}
\label{eq:pso_problem}
\boldsymbol{\alpha}^\star
\in \arg\min_{\boldsymbol{\alpha}\in\mathcal{P}} \mathcal{J}(\boldsymbol{\alpha}),
\qquad
\mathcal{J}(\boldsymbol{\alpha})
\coloneqq
\mathcal{D}\!\left[S^\star(\boldsymbol{\alpha})\right],
\end{equation}
where, for each $\boldsymbol{\alpha}\in\mathcal{P}$, the surface $S^\star(\boldsymbol{\alpha})$ is
the unique Dirichlet extremal in the GT-B\'ezier tensor-product space with the prescribed boundary control
net, obtained by solving
\begin{equation}
\label{eq:pso_inner_system}
A(\boldsymbol{\alpha})\,\mathbf{p}(\boldsymbol{\alpha})=\mathbf{b}(\boldsymbol{\alpha}).
\end{equation}

\subsection{Particle Swarm Optimization (PSO)}
PSO is a population-based stochastic metaheuristic inspired by collective motion such as bird
flocking \cite{jain2022pso,zhang2015pso}. It maintains a set (swarm) of $N$ particles,
each representing a candidate position vector
\[
x_i(t)\in\mathbb{R}^D,\qquad D=4,
\]
where we identify $x_i=(\alpha_1,\alpha_2,\beta_1,\beta_2)$ and enforce the bound constraints
$x_i(t)\in\mathcal{P}$ (e.g.\ via clamping or projection). At iteration $t$, each particle has a
velocity $v_i(t)\in\mathbb{R}^D$, a personal best position $x_i^b(t)$ (the best position it has
visited so far), and the swarm shares a global best position $x_g^b(t)$ (the best across all
particles). With i.i.d.\ random vectors $r_1(t),r_2(t)\in[0,1]^D$, the global-best PSO update is
\begin{equation}
\label{eq:Pso}
\begin{cases}
v_i(t) = w\,v_i(t-1)
+ c_1\,r_1(t)\odot\big(x_g^{b}(t-1)-x_i(t-1)\big)
+ c_2\,r_2(t)\odot\big(x_i^{b}(t-1)-x_i(t-1)\big),\\[1mm]
x_i(t) = x_i(t-1) + v_i(t),
\end{cases}
\end{equation}
where $\odot$ denotes componentwise multiplication. The inertia weight $w$ controls momentum
(exploration vs.\ exploitation), and $c_1,c_2>0$ are acceleration (learning) parameters weighting
the global (social) and personal (cognitive) terms, respectively. In our implementation, after the
position update we enforce feasibility $x_i(t)\in\mathcal{P}$ and then re-evaluate the fitness
$f(x_i(t))\coloneqq \mathcal{J}(x_i(t))$.

\paragraph{Local-best variant.}
A common alternative replaces the global best $x_g^b$ by a neighborhood best $x_i^b$ computed over
a prescribed neighborhood of particle $i$. This can enhance exploration and reduce premature
convergence, at the cost of slower convergence rates. In this work we use the global-best variant
\eqref{eq:Pso} unless otherwise stated.

In our setting the PSO objective is the reduced Dirichlet functional \eqref{eq:pso_problem}, i.e.
\begin{equation}
\label{eq:psoobj}
\min_{\boldsymbol{\alpha}\in\mathcal{P}} f(\boldsymbol{\alpha})
\;=\;
\min_{\boldsymbol{\alpha}\in\mathcal{P}} \mathcal{J}(\boldsymbol{\alpha})
\;=\;
\min_{\boldsymbol{\alpha}\in\mathcal{P}}
\frac12\int_{[0,1]^2}\Big(\|S_u^\star(\boldsymbol{\alpha})\|^2+\|S_v^\star(\boldsymbol{\alpha})\|^2\Big)\,du\,dv.
\end{equation}
Equivalently, one may write $S^\star(\boldsymbol{\alpha}) =
S_{m,n}(\cdot,\cdot;\boldsymbol{\alpha},\mathbf{P}(\boldsymbol{\alpha}))$ with
$\mathbf{P}(\boldsymbol{\alpha})$ obtained from \eqref{eq:pso_inner_system}. Algorithm~\ref{alg:pso} summarizes the PSO procedure used in the numerical experiments.

\begin{algorithm}[b!]
\caption{Particle Swarm Optimization (PSO) for minimizing $\mathcal{J}(\boldsymbol{\alpha})$}
\label{alg:pso}
\begin{algorithmic}[1]
\Require Objective $f(\boldsymbol{\alpha})=\mathcal{J}(\boldsymbol{\alpha})$, admissible set $\mathcal{P}$,
dimension $D=4$, swarm size $N$, max iterations $T_{\max}$, parameters $w,c_1,c_2$
\Ensure Best parameters $\boldsymbol{\alpha}_g^b$ and value $f(\boldsymbol{\alpha}_g^b)$
\State Initialize particles $\boldsymbol{\alpha}_i(0)\in\mathcal{P}$ (e.g.\ uniform in bounds), $i=1,\dots,N$
\State Initialize velocities $v_i(0)$ (e.g.\ uniform in a symmetric box), $i=1,\dots,N$
\State Set personal bests $\boldsymbol{\alpha}_i^b(0)=\boldsymbol{\alpha}_i(0)$ and evaluate $f(\boldsymbol{\alpha}_i^b(0))$
\State Set global best $\boldsymbol{\alpha}_g^b(0)\in\arg\min_i f(\boldsymbol{\alpha}_i^b(0))$
\For{$t=1$ to $T_{\max}$}
  \For{$i=1$ to $N$}
    \State Draw $r_1,r_2\sim U([0,1]^D)$
    \State Update velocity:
    $v_i(t)=w\,v_i(t-1)+c_1\,r_1\odot(\boldsymbol{\alpha}_g^b(t-1)-\boldsymbol{\alpha}_i(t-1))
    +c_2\,r_2\odot(\boldsymbol{\alpha}_i^b(t-1)-\boldsymbol{\alpha}_i(t-1))$
    \State Update position: $\boldsymbol{\alpha}_i(t)=\boldsymbol{\alpha}_i(t-1)+v_i(t)$
    \State Enforce bounds: project/clamp $\boldsymbol{\alpha}_i(t)$ onto $\mathcal{P}$
    \State Evaluate fitness $f(\boldsymbol{\alpha}_i(t))$
    \If{$f(\boldsymbol{\alpha}_i(t)) < f(\boldsymbol{\alpha}_i^b(t-1))$}
      \State $\boldsymbol{\alpha}_i^b(t)=\boldsymbol{\alpha}_i(t)$
      \If{$f(\boldsymbol{\alpha}_i^b(t)) < f(\boldsymbol{\alpha}_g^b(t-1))$}
        \State $\boldsymbol{\alpha}_g^b(t)=\boldsymbol{\alpha}_i^b(t)$
      \EndIf
    \Else
      \State $\boldsymbol{\alpha}_i^b(t)=\boldsymbol{\alpha}_i^b(t-1)$
    \EndIf
  \EndFor
\EndFor
\State \Return $\boldsymbol{\alpha}_g^b(T_{\max})$, $f(\boldsymbol{\alpha}_g^b(T_{\max}))$
\end{algorithmic}
\end{algorithm}

Unless otherwise specified, we use a swarm size $N=50$, inertia weight $w=0.7$, and acceleration
coefficients $c_1=c_2=1.5$. Each particle is initialized as a 4-vector
$\boldsymbol{\alpha}_i(0)=(\alpha_1,\alpha_2,\beta_1,\beta_2)\in\mathcal{P}$ (uniformly at random
within the bounds of $\mathcal{P}$), and velocities are initialized uniformly in a symmetric range
proportional to the parameter box size. The global-best variant \eqref{eq:Pso} is employed.

\section{Numerical examples: Plateau problem for GT--B\'ezier patches via Dirichlet extremals}
\label{sec:examples}

In this section we report numerical experiments for the (discrete) Plateau problem in the class of
GT--B\'ezier tensor-product patches. Throughout, the boundary control net is prescribed, and the
unknown interior control points are computed as the \emph{Dirichlet extremal} in the chosen finite-dimensional
GT--B\'ezier space (Section~\ref{sec:dirichlet-extremals}, cf.\ \eqref{eq:param_linear_system}). In addition, the shape-parameter
vector $\boldsymbol{\alpha}=(\alpha_1,\alpha_2,\beta_1,\beta_2)\in\mathcal{P}$ is selected by minimizing the
\emph{reduced} Dirichlet energy
\[
\mathcal{J}(\boldsymbol{\alpha})
\;=\;
\mathcal{D}\!\left[S^\star(\boldsymbol{\alpha})\right],
\qquad
\boldsymbol{\alpha}^\star \in \arg\min_{\boldsymbol{\alpha}\in\mathcal{P}} \mathcal{J}(\boldsymbol{\alpha}),
\]
where $S^\star(\boldsymbol{\alpha})$ denotes the GT--Dirichlet extremal surface for the fixed boundary data at
parameter $\boldsymbol{\alpha}$ (Section~\ref{sec:pso}). Hence each evaluation $\mathcal{J}(\boldsymbol{\alpha})$
requires solving the interior linear system and then computing $\mathcal{D}$ (numerically) on $[0,1]^2$.

For comparison we also construct: (i) the classical Bernstein--B\'ezier Dirichlet patch
\cite{monterde2004bezier2}, (ii) the quasi-harmonic surface \cite{hao2012minimal}, and (iii) the bending-energy
surface \cite{miao2005bezier}. All surfaces share the same boundary control points. We visualize the resulting
patches and, in addition, we display mean-curvature nephograms. For a smooth minimal surface the mean curvature
$H$ vanishes identically; in discrete or constrained settings one typically expects $H$ to be small rather than
exactly zero \cite{monterde2003Plateau}.

\begin{example}[Bicubic boundary net I]
\label{ex:quaddirichlet}
Let $\{P_{ij}\}_{i,j=0}^3\subset\mathbb{R}^3$ be a $4\times 4$ bicubic control net whose boundary points are fixed as
\[
\begin{aligned}
&P_{0,0}=(0,0,0),\quad P_{0,1}=(2,0,2),\quad P_{0,2}=(4,0,-2),\quad P_{0,3}=(6,0,0),\\
&P_{1,0}=(0,2,2),\quad P_{1,3}=(6,2,2),\\
&P_{2,0}=(0,4,-2),\quad P_{2,3}=(6,4,-2),\\
&P_{3,0}=(0,6,2),\quad P_{3,1}=(2,6,2),\quad P_{3,2}=(4,6,-2),\quad P_{3,3}=(6,6,0).
\end{aligned}
\]
For each admissible shape-parameter vector $\boldsymbol{\alpha}=(\alpha_1,\alpha_2,\beta_1,\beta_2)\in\mathcal{P}$,
let $S^\star(\cdot,\cdot;\boldsymbol{\alpha})$ denote the GT--B\'ezier patch whose boundary control points are
as above and whose interior control points are chosen as the Dirichlet extremal, i.e.\ obtained from
\eqref{eq:param_linear_system}. The numerical Plateau problem in this GT--B\'ezier class is then the reduced optimization
problem
\[
\boldsymbol{\alpha}^\star\in\arg\min_{\boldsymbol{\alpha}\in\mathcal{P}}
\mathcal{J}(\boldsymbol{\alpha})
=
\arg\min_{\boldsymbol{\alpha}\in\mathcal{P}}
\mathcal{D}\!\left[S^\star(\boldsymbol{\alpha})\right].
\]
We solve this problem by PSO as described in Section~\ref{sec:pso}. The resulting surface (with the optimized
shape parameters) is shown in Figure~\ref{fig:dirichletquad}(a). The competing surfaces obtained by the
Bernstein--B\'ezier Dirichlet method, the quasi-harmonic method, and the bending-energy method are shown in
Figure~\ref{fig:dirichletquad}(b)--(d). The PSO convergence histories (10 independent runs) are shown in
Figure~\ref{fig:areavsiterquad}(a).
\end{example}

\begin{figure}[htb]
\begin{center}
\subfigure[Extremal Dirichlet method in the GT class (PSO-optimized $\boldsymbol{\alpha}$).]{
\includegraphics[width=2.5in]{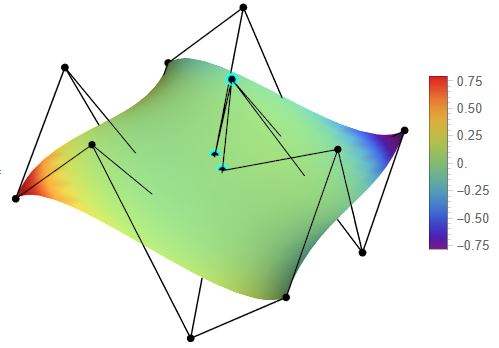}}
\subfigure[Classical Bernstein--B\'ezier Dirichlet method \cite{monterde2004bezier2}.]{
\includegraphics[width=2.5in]{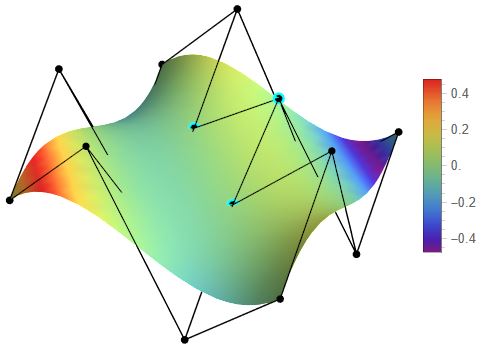}}\\
\subfigure[Quasi-harmonic method \cite{hao2012minimal}.]{
\includegraphics[width=2.5in]{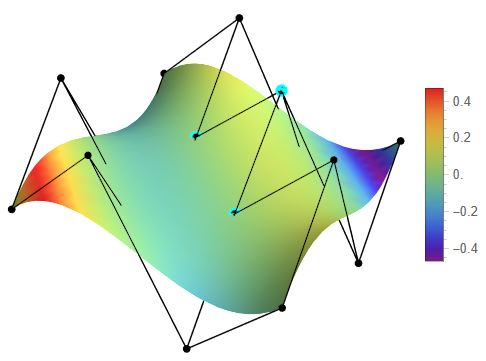}}
\subfigure[Bending-energy method \cite{miao2005bezier}.]{
\includegraphics[width=2.5in]{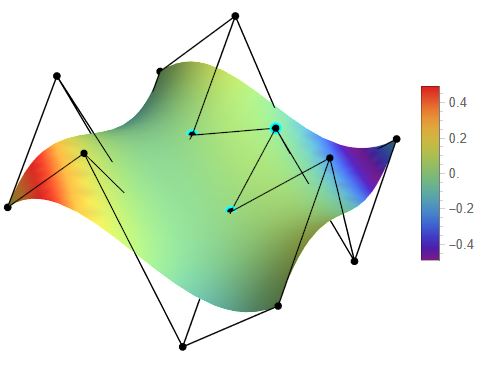}}
\caption{Bicubic surfaces constructed from the same boundary control points by different approaches.}
\label{fig:dirichletquad}
\end{center}
\end{figure}
\begin{example}[Bicubic boundary net II]
\label{ex:dirichletcubic}
Let the boundary control points of a bicubic net $\{P_{ij}\}_{i,j=0}^3$ be prescribed by
\[
\begin{aligned}
&P_{0,0}=(0,0,0),\quad P_{0,1}=(2,0,2),\quad P_{0,2}=(4,0,2),\quad P_{0,3}=(6,0,0),\\
&P_{1,0}=(0,2,2),\quad P_{1,3}=(6,2,2),\\
&P_{2,0}=(0,4,2),\quad P_{2,3}=(6,4,2),\\
&P_{3,0}=(0,6,2),\quad P_{3,1}=(2,6,2),\quad P_{3,2}=(4,6,2),\quad P_{3,3}=(6,6,0).
\end{aligned}
\]
As in Example~\ref{ex:quaddirichlet}, for each $\boldsymbol{\alpha}\in\mathcal{P}$ we compute the interior control
points as the GT--Dirichlet extremal and minimize the reduced functional $\mathcal{J}(\boldsymbol{\alpha})$ by PSO.
The resulting GT surface is displayed in Figure~\ref{fig:dirichletcubic}(a), together with the comparison methods
in Figure~\ref{fig:dirichletcubic}(b)--(d). The convergence histories for 10 runs are shown in
Figure~\ref{fig:areavsiterquad}(b). In particular, the proposed method yields a surface with small mean curvature
magnitudes (nearly harmonic behavior) while achieving the lowest surface area/Dirichlet energy among the tested
methods for this boundary configuration.
\end{example}

\begin{figure}[t!]
\begin{center}
\subfigure[Extremal Dirichlet method in the GT class (PSO-optimized $\boldsymbol{\alpha}$).]{
\includegraphics[width=2.7in]{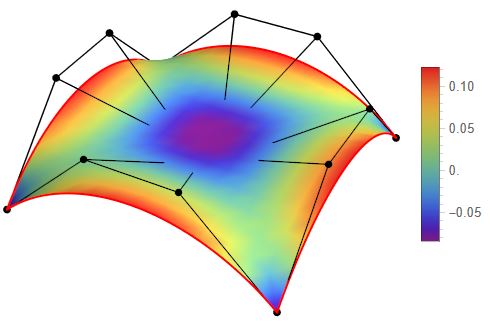}}
\subfigure[Classical Bernstein--B\'ezier Dirichlet method \cite{monterde2004bezier2}.]{
\includegraphics[width=2.7in]{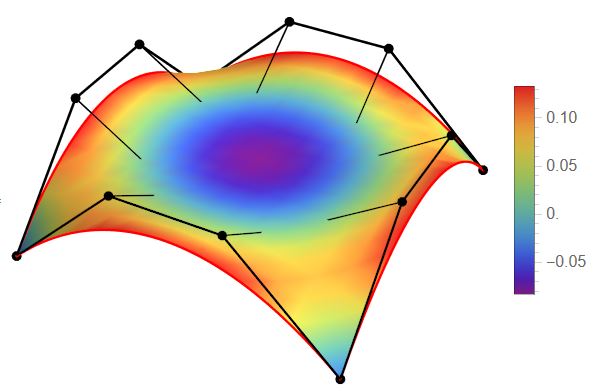}}\\
\subfigure[Quasi-harmonic method \cite{hao2012minimal}.]{
\includegraphics[width=2.7in]{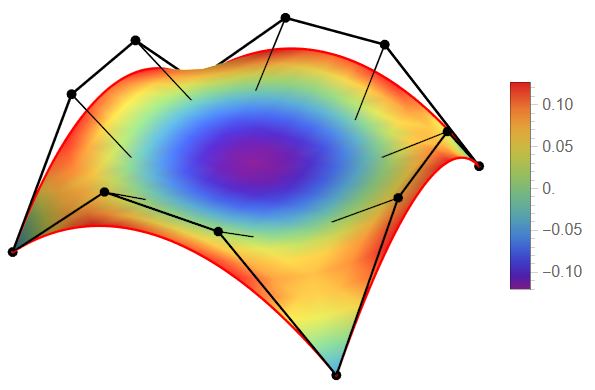}}
\subfigure[Bending-energy method \cite{miao2005bezier}.]{
\includegraphics[width=2.7in]{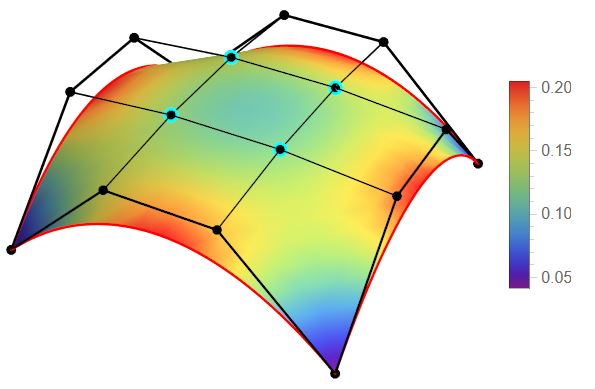}}
\caption{Bicubic surfaces constructed from the same boundary control points by different approaches.}
\label{fig:dirichletcubic}
\end{center}
\end{figure}

\begin{figure}[t!]
\begin{center}
\subfigure[]{\includegraphics[width=2.7in]{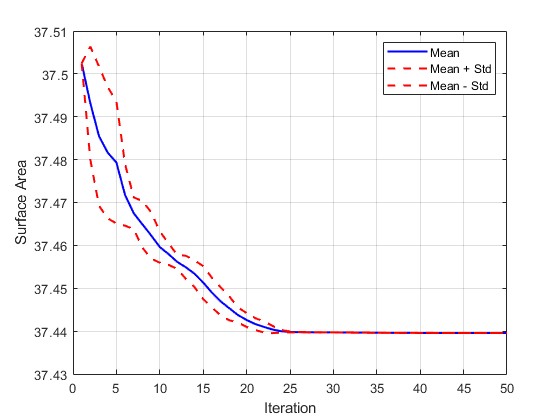}}
\subfigure[]{\includegraphics[width=2.7in]{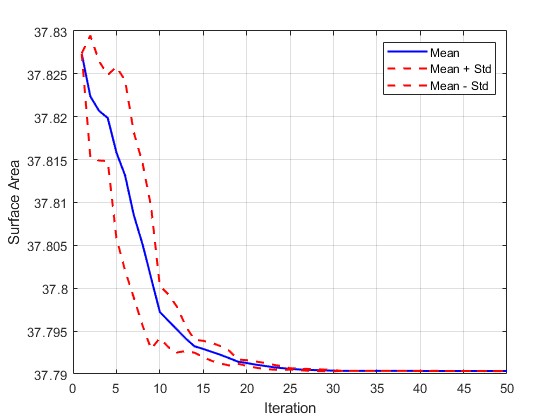}}
\caption{PSO convergence histories for minimizing $\mathcal{J}(\boldsymbol{\alpha})$ (10 runs):
(a) Example~\ref{ex:quaddirichlet}; (b) Example~\ref{ex:dirichletcubic}.}
\label{fig:areavsiterquad}
\end{center}
\end{figure}

To quantify the comparison, Table~\ref{tab:one} reports the attained surface areas (computed from the resulting
parametric surfaces) together with the optimized shape parameters for the proposed method. In both examples the
PSO-optimized GT--Dirichlet extremal yields the smallest area among the tested approaches, indicating that the
added shape degrees of freedom (optimized at the energy level) effectively improve the quality of the discrete
Plateau solution within the chosen finite-dimensional ansatz space.

\begin{table}[t!]
\centering
\caption{Surface-area comparison for Examples~\ref{ex:quaddirichlet} and~\ref{ex:dirichletcubic}.}
\vspace{0.5em}
\begin{tabular}{|p{3.6cm}|c|p{3.8cm}|c|p{3.8cm}|}
\hline
Method &
\multicolumn{2}{c|}{Example~\ref{ex:quaddirichlet}} &
\multicolumn{2}{c|}{Example~\ref{ex:dirichletcubic}} \\
\cline{2-5}
& Area & Shape parameters & Area & Shape parameters \\
\hline
PSO-based GT--Dirichlet extremal (proposed) &
\underline{37.4396} &
$\alpha_1=\alpha_2=\beta_1=\beta_2=0.8706$ &
\underline{37.7905} &
$\alpha_1=\alpha_2=\beta_1=\beta_2=1.4823$ \\
\hline
Bernstein--B\'ezier Dirichlet method \cite{monterde2004bezier2} &
38.0000 & N/A & 38.0000 & N/A \\
\hline
Quasi-harmonic method \cite{hao2012minimal} &
38.3817 & N/A & 38.0826 & N/A \\
\hline
Bending-energy method \cite{miao2005bezier} &
38.4057 & N/A & 40.0661 & N/A \\
\hline
\end{tabular}
\label{tab:one}
\end{table}


\section{Relation to harmonic patches}
\label{sec:harmonic}

Let $D=[0,1]^2$ be the parameter domain. For a sufficiently smooth parametric surface
$S\in H^1(D;\mathbb{R}^3)$, the Dirichlet functional is
\begin{equation}
\label{eq:dirichlet_general}
\mathcal{D}(S)
=
\frac12\int_{D}\Big(\|S_u\|^2+\|S_v\|^2\Big)\,du\,dv .
\end{equation}
In the \emph{unconstrained} (infinite-dimensional) setting, where $S$ ranges over all admissible maps with fixed
boundary trace $S|_{\partial D}$---the Euler--Lagrange equation associated with \eqref{eq:dirichlet_general} is the
vector Laplace equation
\begin{equation}
\label{eq:laplace}
\Delta S \;=\; S_{uu}+S_{vv} \;=\; 0
\qquad \text{in }D,
\end{equation}
so that the minimizer is harmonic. Denote this unconstrained minimizer by $S^{\mathrm{harm}}$.

In the \emph{constrained} (finite-dimensional) setting of this paper, $S$ is restricted to the GT--B\'ezier
tensor-product space with prescribed boundary control points (and, for each fixed
$\boldsymbol{\alpha}=(\alpha_1,\alpha_2,\beta_1,\beta_2)\in\mathcal{P}$, the interior control points are chosen as
the Dirichlet extremal in that space). Let $S^{\mathrm{GT}}$ be the resulting constrained minimizer for the same
boundary data. Since the admissible class is smaller, the variational inequality
\begin{equation}
\label{eq:dirichlet_inequality}
\mathcal{D}(S^{\mathrm{harm}})\;\le\;\mathcal{D}(S^{\mathrm{GT}})
\end{equation}
holds, with strict inequality in general. Equality can occur only if the constrained patch happens to satisfy
\eqref{eq:laplace}, i.e.\ it is harmonic in the classical sense.

\subsection{Constructing (approximately) harmonic GT--B\'ezier surfaces}
\label{sec:harmonic_construct}

Because a generic GT--B\'ezier patch does not satisfy $\Delta S\equiv 0$ identically, we measure harmonicity
by the least-squares Laplacian defect
\begin{equation}
\label{eq:obj_harmonic}
\mathcal{F}(\boldsymbol{\alpha})
=
\int_{D}\big\|\Delta S(u,v;\boldsymbol{\alpha})\big\|^2\,du\,dv,
\qquad
\boldsymbol{\alpha}=(\alpha_1,\alpha_2,\beta_1,\beta_2)\in\mathcal{P},
\end{equation}
and select $\boldsymbol{\alpha}$ by minimizing $\mathcal{F}$ (using PSO as in Section~\ref{sec:pso}). Here
$S(\cdot,\cdot;\boldsymbol{\alpha})$ is the GT--B\'ezier surface constructed from a control net whose missing
points are first computed from a discrete harmonicity condition at the control-net level, following
\cite{monterde2003Plateau}.

Let $\widetilde S(u,v)=\sum_{i=0}^{n}\sum_{j=0}^{m}B_i^n(u)B_j^m(v)\,P_{ij}$ be a classical tensor-product
B\'ezier surface. Using standard second-difference operators
\[
\Delta^{2,0}P_{ij}=P_{i+2,j}-2P_{i+1,j}+P_{i,j},
\qquad
\Delta^{0,2}P_{ij}=P_{i,j+2}-2P_{i,j+1}+P_{i,j},
\]
one can express $\Delta \widetilde S$ in a Bernstein basis of degree $(n,m)$ and obtain, for interior indices
$(i,j)$, a linear relation of the form (see \cite{monterde2003Plateau})

\begin{equation}
    \begin{aligned}
    &n(n-1)\Big(a_{i,n}\Delta^{2,0}P_{ij}+b_{i-1,n}\Delta^{2,0}P_{i-1,j}+c_{i-2,n}\Delta^{2,0}P_{i-2,j}\Big)
    \\
&+m(m-1)\Big(a_{j,m}\Delta^{0,2}P_{ij}+b_{j-1,m}\Delta^{0,2}P_{i,j-1}+c_{j-2,m}\Delta^{0,2}P_{i,j-2}\Big)=0,
\end{aligned}
\label{eq:pointsharmonic}
\end{equation}

where, for $k\in\{0,\dots,n-2\}$,
\[
a_{k,n}=(n-k)(n-k-1),\qquad
b_{k,n}=2(k+1)(n-k-1),\qquad
c_{k,n}=(k+1)(k+2),
\]
and similarly for $(a_{k,m},b_{k,m},c_{k,m})$.
Given partial boundary data, \eqref{eq:pointsharmonic} yields a linear system for the missing control points of
a harmonic (Bernstein) patch.

After computing the missing control points by \eqref{eq:pointsharmonic} (in the Bernstein setting), we construct
a GT--B\'ezier patch with the same control net and then tune $\boldsymbol{\alpha}$ by minimizing the
harmonicity defect \eqref{eq:obj_harmonic}. This yields a GT patch with $\Delta S$ small in the $L^2$ sense,
hence ``approximately harmonic'' in the variational sense.

\begin{remark}
Classical B\'ezier patches admit harmonic constructions through linear constraints on the control net, but they
lack the additional shape degrees of freedom inherent to the GT-B\'ezier basis. The GT-B\'ezier basis parameters provide an extra level
of geometric control: they can be selected to (i) reduce the harmonicity defect \eqref{eq:obj_harmonic} and/or
(ii) subsequently reduce the Dirichlet energy (or area) while maintaining the same boundary data.
\end{remark}

\begin{example}[Harmonic reconstruction from partial boundary data]
\label{ex:harmonic}
We consider two configurations of missing data. In \emph{Case 1}, only the first and last columns of the control
net are prescribed (Figure~\ref{fig:harmonic}(a)); in \emph{Case 2}, only the first and last rows are prescribed
(Figure~\ref{fig:harmonic}(c)). In each case we compute the remaining control points by solving the linear system
induced by \eqref{eq:pointsharmonic}. We then construct the corresponding GT--B\'ezier patch and optimize the
shape parameters $\boldsymbol{\alpha}\in\mathcal{P}$ by minimizing the harmonicity defect
$\mathcal{F}(\boldsymbol{\alpha})$ in \eqref{eq:obj_harmonic} via PSO. The resulting approximately harmonic GT-B\'ezier 
surfaces and their curvature nephograms are shown in Figure~\ref{fig:harmonic}(b),(d). The PSO convergence history
(10 runs) is reported in Figure~\ref{fig:harmonic_convergence}.
\end{example}

\begin{figure}[t!]
\begin{center}
\subfigure[]{\includegraphics[width=2.5in]{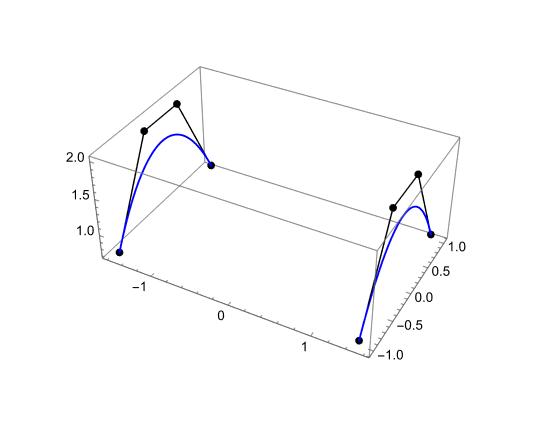}}\hspace{5em}
\subfigure[]{\includegraphics[width=2.7in]{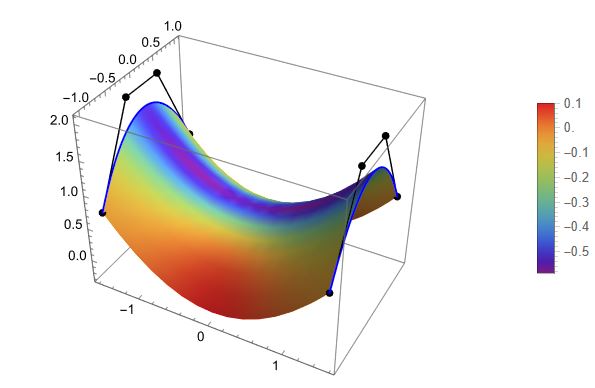}}\\
\subfigure[]{\includegraphics[width=2.5in]{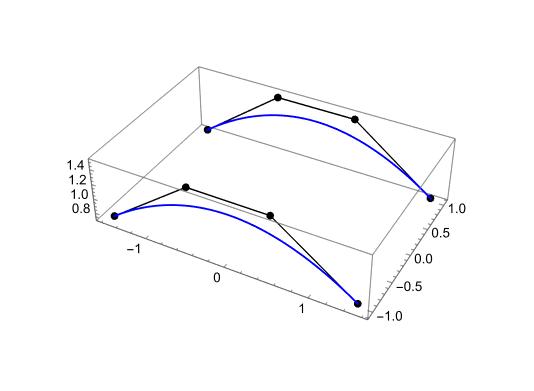}}\hspace{5em}
\subfigure[]{\includegraphics[width=2.7in]{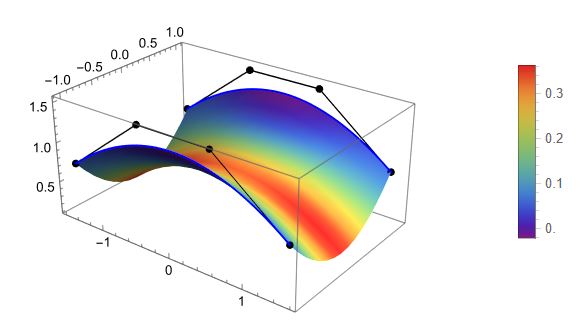}}
\caption{Approximately harmonic GT--B\'ezier surfaces from partial boundary control data.}
\label{fig:harmonic}
\end{center}
\end{figure}

\begin{figure}[t!]
\begin{center}
\includegraphics[width=4in]{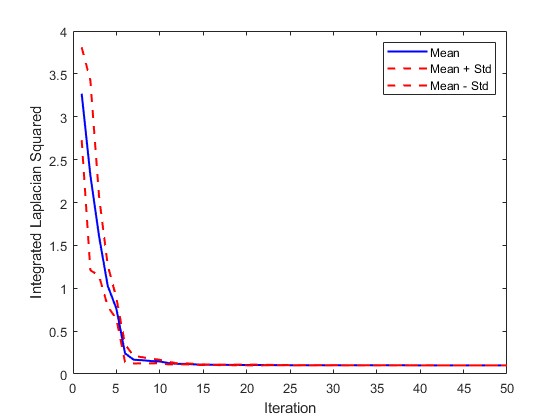}
\caption{PSO convergence for minimizing the harmonicity defect $\mathcal{F}(\boldsymbol{\alpha})$ in
\eqref{eq:obj_harmonic}.}
\label{fig:harmonic_convergence}
\end{center}
\end{figure}
\section{Extremals of the Dirichlet functional and minimal TB--Coons surfaces}
\label{sec:coon}

Let $D=[0,1]^2$ be the parameter domain. Given four boundary curves
\[
\Gamma_0(u),\Gamma_1(u)\ (u\in[0,1]),\qquad \Gamma_2(v),\Gamma_3(v)\ (v\in[0,1]),
\]
satisfying the standard vertex compatibility conditions
\[
\Gamma_0(0)=\Gamma_2(0),\ \Gamma_0(1)=\Gamma_3(0),\ \Gamma_1(0)=\Gamma_2(1),\ \Gamma_1(1)=\Gamma_3(1),
\]
the classical bilinearly blended Coons patch is the interpolant $\widetilde S:D\to\mathbb{R}^3$ defined by
\begin{equation}
\label{eq:coons_classical}
\widetilde S(u,v)= (1-v)\Gamma_0(u)+v\Gamma_1(u) + (1-u)\Gamma_2(v)+u\Gamma_3(v)
-\widetilde S_{\mathrm{bil}}(u,v),
\end{equation}
where the bilinear corner correction is
\begin{equation}
\label{eq:bilinear_corner}
\widetilde S_{\mathrm{bil}}(u,v)
=(1-u)(1-v)\,C_{00}+(1-u)v\,C_{01}+u(1-v)\,C_{10}+uv\,C_{11},
\end{equation}
with $C_{00}=\Gamma_0(0)$, $C_{01}=\Gamma_0(1)$, $C_{10}=\Gamma_1(0)$, $C_{11}=\Gamma_1(1)$.
Equivalently, \eqref{eq:coons_classical} can be written in the compact matrix form
\begin{equation}
\label{eq:coons_matrix}
\widetilde S(u,v)=
-\begin{bmatrix}-1 & 1-u & u\end{bmatrix}
\begin{bmatrix}
0 & \Gamma_0(u) & \Gamma_1(u)\\
\Gamma_2(v) & C_{00} & C_{01}\\
\Gamma_3(v) & C_{10} & C_{11}
\end{bmatrix}
\begin{bmatrix}-1\\ 1-v\\ v\end{bmatrix}.
\end{equation}
Throughout, we use the standard linear blending functions
\[
h_0(u)=1-u,\quad h_1(u)=u,\qquad g_0(v)=1-v,\quad g_1(v)=v,
\]
so that \eqref{eq:coons_matrix} matches the usual Coons construction.

Assume that the four boundary curves are represented by cubic control polygons. Using the notation of previous
sections, let $\{B_{i,3}\}_{i=0}^3$ denote the classical Bernstein basis and let $\{G_{i,3}(\cdot;\boldsymbol{\alpha})\}_{i=0}^3$
denote the cubic GT--B\'ezier basis, with parameter vector
$\boldsymbol{\alpha}=(\alpha_1,\alpha_2,\beta_1,\beta_2)\in\mathcal{P}$ (the same admissible set $\mathcal{P}$ used in
Section~\ref{sec:pso}). Let $\{Q_{ij}\}_{i,j=0}^3\subset\mathbb{R}^3$ be the $4\times4$ control net, where the twelve boundary
points are prescribed and the four interior points $Q_{11},Q_{12},Q_{21},Q_{22}$ are unknown. We represent the $u$-boundary curves by cubic Bernstein polynomials
\begin{equation}
\label{eq:coon_curve_u}
\Gamma_k(u)=\sum_{i=0}^{3} B_{i,3}(u)\,Q_{k,i},\qquad k\in\{0,3\},\quad u\in[0,1],
\end{equation}
and the $v$-boundary curves similarly as
\begin{equation}
\label{eq:coon_curve_v}
\Gamma_\ell(v)=\sum_{j=0}^{3} B_{j,3}(v)\,Q_{j,\ell},\qquad \ell\in\{0,3\},\quad v\in[0,1],
\end{equation}
where the index placement is chosen to be consistent with the tensor-product net $\{Q_{ij}\}$ (corners are
$Q_{00},Q_{03},Q_{30},Q_{33}$).  The Bernstein basis functions $B_{i,3}$ are as in \eqref{eq:bernstein}.
(Compared to the earlier discussion, we avoid mixing $m,n$ with a fixed cubic degree, and we correct the index mismatch.)

We now introduce two tensor-product patches that share the same control net $\{Q_{ij}\}$ but use mixed basis. 

(i) Bernstein in $u$ and GT--B\'ezier in $v$.
Define
\begin{equation}
\label{eq:R1}
 R_{1}(u,v;\boldsymbol{\alpha})
=\sum_{i=0}^{3}\sum_{j=0}^{3} B_{i,3}(u)\,G_{j,3}(v;\boldsymbol{\alpha})\,Q_{ij}.
\end{equation}
Then the boundary curves $v=0$ and $v=1$ interpolate the Bernstein curves determined by the first and last rows
$\{Q_{i0}\}_{i=0}^3$ and $\{Q_{i3}\}_{i=0}^3$, while $u=0,1$ are GT-B\'ezier curves in $v$. 

(ii) GT--B\'ezier in $u$ and Bernstein in $v$.
Similarly,
\begin{equation}
\label{eq:R2}
 R_{2}(u,v;\boldsymbol{\alpha})
=\sum_{i=0}^{3}\sum_{j=0}^{3} G_{i,3}(u;\boldsymbol{\alpha})\,B_{j,3}(v)\,Q_{ij}.
\end{equation}

Construct the four GT boundary curves (degree $3$) induced by the same boundary control points:
\begin{equation}
\label{eq:GT_boundary_u}
\Gamma'_k(u;\boldsymbol{\alpha})=\sum_{i=0}^{3} G_{i,3}(u;\boldsymbol{\alpha})\,Q_{k,i},\qquad k\in\{0,3\},
\end{equation}
\begin{equation}
\label{eq:GT_boundary_v}
\Gamma'_\ell(v;\boldsymbol{\alpha})=\sum_{j=0}^{3} G_{j,3}(v;\boldsymbol{\alpha})\,Q_{j,\ell},\qquad \ell\in\{0,3\}.
\end{equation}
Using these four GT curves in the bilinear Coons blending, define the GT--Coons corner correction patch
\begin{equation}
\label{eq:T}
 T(u,v;\boldsymbol{\alpha})
=
-\begin{bmatrix}-1 & 1-u & u\end{bmatrix}
\begin{bmatrix}
0 & \Gamma'_0(u;\boldsymbol{\alpha}) & \Gamma'_3(u;\boldsymbol{\alpha})\\
\Gamma'_0(v;\boldsymbol{\alpha}) & Q_{00} & Q_{03}\\
\Gamma'_3(v;\boldsymbol{\alpha}) & Q_{30} & Q_{33}
\end{bmatrix}
\begin{bmatrix}-1\\ 1-v\\ v\end{bmatrix}.
\end{equation}
By construction, $ T$ coincides with the same corner data and removes the double-counting of the bilinear part
in the standard Coons decomposition.

We define the TB--Coons surface (tensor-product Bernstein/GT hybrid with Coons correction) by
\begin{equation}
\label{eq:Suv}
{\mathcal S}(u,v;\boldsymbol{\alpha})
=
 R_{1}(u,v;\boldsymbol{\alpha})
+ R_{2}(u,v;\boldsymbol{\alpha})
- T(u,v;\boldsymbol{\alpha}).
\end{equation}
The surface ${\mathcal S}$ interpolates the prescribed boundary curves in the sense inherited from
\eqref{eq:R1}--\eqref{eq:T}. The construction is illustrated in Fig.~\ref{fig:coon_patch_1234}.  The remaining degrees of freedom are precisely the four interior control points
$Q_{11},Q_{12},Q_{21},Q_{22}$ (and, if desired, the shape parameters $\boldsymbol{\alpha}$).

To impose minimality, we minimize the
Dirichlet energy on $D$:
\begin{equation}
\label{eq:coon_dirichlet}
\mathcal{D}\!\left({\mathcal S}(\cdot,\cdot;\boldsymbol{\alpha})\right)
=
\frac12\int_{D}
\left(\left\|\partial_u{\mathcal S}(u,v;\boldsymbol{\alpha})\right\|^2
+\left\|\partial_v{\mathcal S}(u,v;\boldsymbol{\alpha})\right\|^2\right)\,du\,dv,
\end{equation}

\begin{figure}[t!]
\begin{center}
\subfigure[]{\includegraphics[width=2.2in]{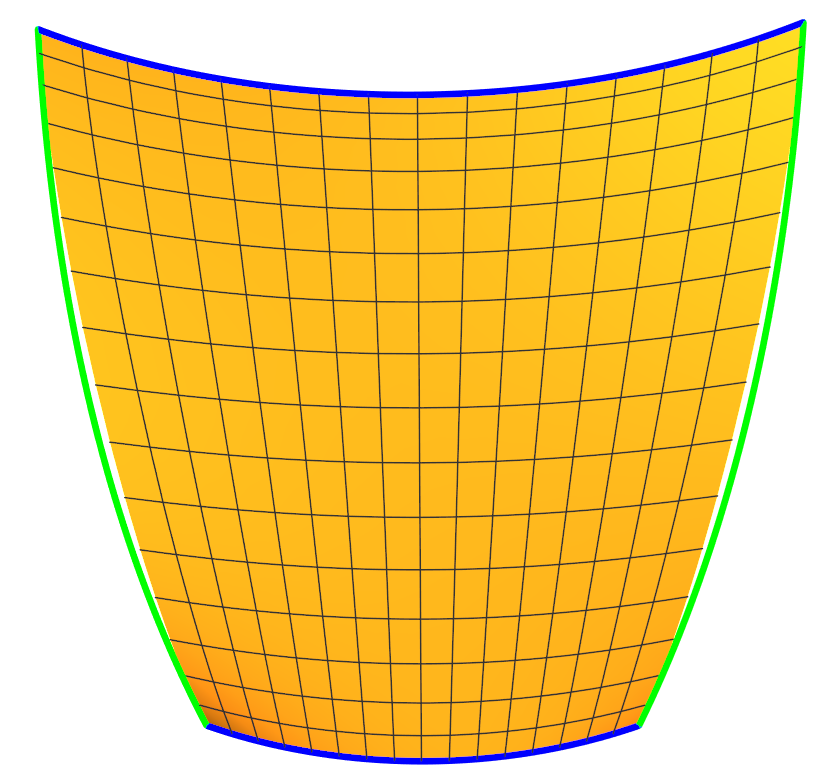}}\hspace{5em}
\subfigure[]{\includegraphics[width=2.2in]{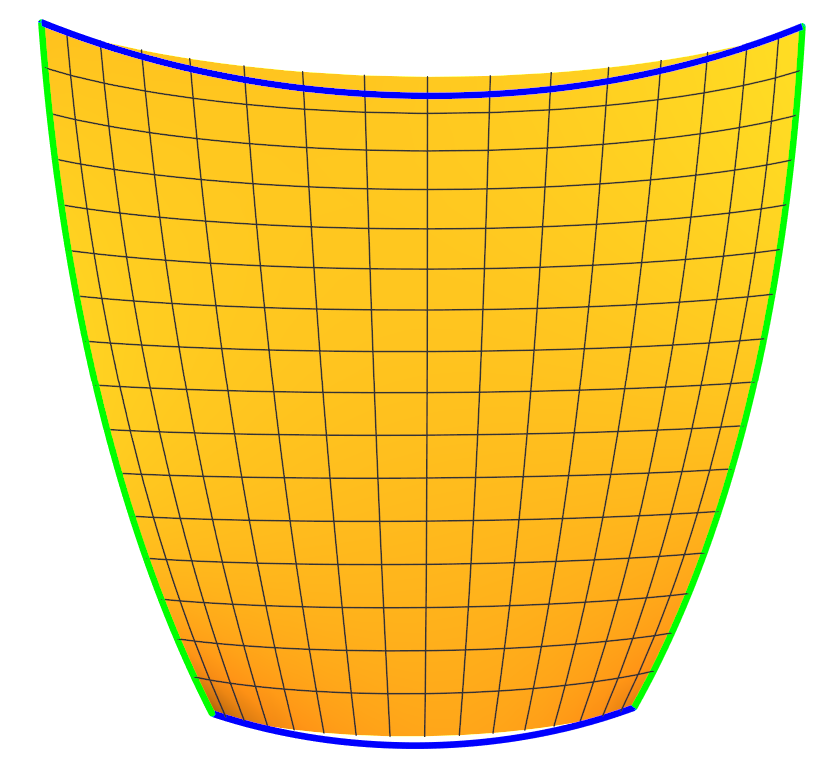}}
\\
\subfigure[]{\includegraphics[width=2.2in]{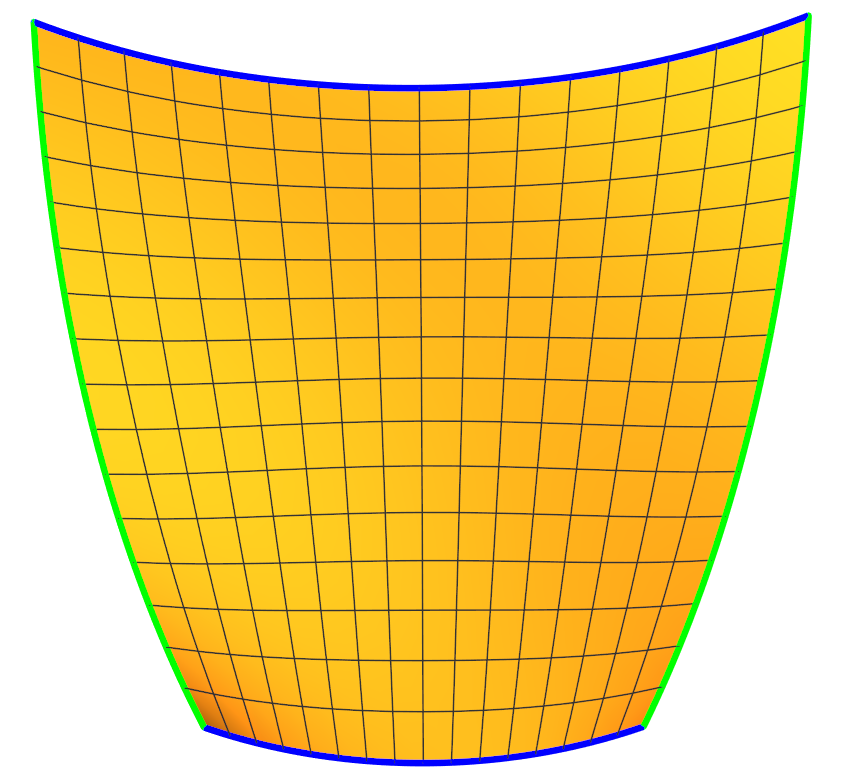}}\hspace{5em}
\subfigure[]{\includegraphics[width=2.in]{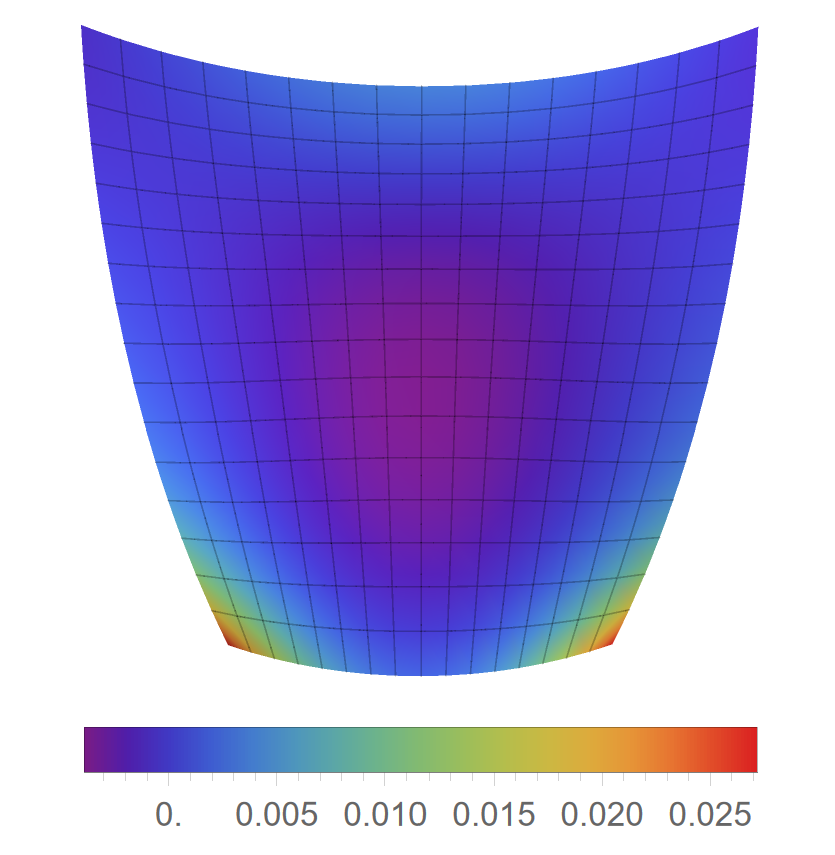}}
\caption{Construction of the TB--Coons surface: (a) $ R_1$ in \eqref{eq:R1}, (b) $ R_2$ in \eqref{eq:R2},
(c) final surface ${\mathcal S}$ in \eqref{eq:Suv}, (d) mean curvature nephogram of ${\mathcal S}$.}
\label{fig:coon_patch_1234}
\end{center}
\end{figure}
subject to fixed boundary control points. In the unconstrained setting, the minimizer with fixed boundary trace is
harmonic (and, under conformality, minimal). Here, however, we restrict to the finite-dimensional TB--Coons ansatz
\eqref{eq:Suv}; thus the computed patch is a \emph{Dirichlet extremal within the TB--Coons space}, i.e. the Galerkin
projection of the harmonic map onto the span of the chosen basis.

For each fixed $\boldsymbol{\alpha}\in\mathcal{P}$, the functional \eqref{eq:coon_dirichlet} is a quadratic form in the
unknown interior control points. Consequently, the necessary and sufficient optimality condition is the linear system
\begin{equation}
\label{eq:coon_linear_system}
\nabla_{(Q_{11},Q_{12},Q_{21},Q_{22})}\,
\mathcal{D}\!\left({\mathcal S}(\cdot,\cdot;\boldsymbol{\alpha})\right)=0,
\end{equation}
which yields the unique interior control points provided the associated stiffness matrix is positive definite
(which holds under the standard nondegeneracy assumptions on the basis).
We additionally optimize $\boldsymbol{\alpha}$ over $\mathcal{P}$ using PSO as in
Section~\ref{sec:pso}, i.e.
\begin{equation}
\label{eq:coon_pso}
\boldsymbol{\alpha}^\star\in\arg\min_{\boldsymbol{\alpha}\in\mathcal{P}}
\mathcal{D}\!\left({\mathcal S}(\cdot,\cdot;\boldsymbol{\alpha})\right),
\end{equation}
where for each particle position $\boldsymbol{\alpha}$ the interior points are computed from \eqref{eq:coon_linear_system}.
We call the resulting surface ${\mathcal S}(\cdot,\cdot;\boldsymbol{\alpha}^\star)$ a \emph{minimal TB--Coons surface}
in the sense of Dirichlet extremality under the TB--Coons constraint.

\section{Conclusion}
\label{sec:conclusion}

This work develops a Plateau-type construction for GT-Bézier surfaces by combining harmonic-relaxation ideas with the additional geometric flexibility provided by the generalized trigonometric basis. With boundary interpolation enforced at the level of the control net, the interior degrees of freedom are determined through a Dirichlet-energy extremal condition within the chosen finite-dimensional GT-B\'ezier  space. For each admissible selection of the GT-B\'ezier basis shape parameters, this variational principle yields a parameter-dependent linear system whose solution specifies the interior control net, and hence a variationally distinguished surface spanning the prescribed boundary data.

A central outcome is a two-level computational strategy: an inner linear solve that produces the Dirichlet extremal surface for fixed parameters, and an outer low-dimensional search that tunes the shape parameters to decrease the Dirichlet energy further. Implemented via particle swarm optimization, this coupling consistently improves the energy in our experiments and, in most tested configurations, leads to a reduction in the realized surface area when compared with classical Bernstein--B\'ezier Dirichlet patches as well as representative quasi-harmonic and bending-energy competitors, all under identical boundary constraints. The observed improvements support the viewpoint that shape-parameterized non-polynomial basis can enhance discrete variational surface construction without altering the prescribed boundary geometry.

We also examined the relationship between harmonicity and minimality within this framework. In the unconstrained setting, Dirichlet-energy minimization selects harmonic parametrizations; in contrast, restricting to GT-B\'ezier tensor-product spaces yields constrained extremals whose deviation from harmonic behavior can be assessed and, when desired, reduced through parameter tuning. Finally, the same extremal Dirichlet principle was transferred to a hybrid TB--Coons construction, providing a systematic way to infer interior control points and generate minimality-biased Coons-type surfaces from sparse boundary information.

Future work includes a sharper analysis of admissible parameter regions and the conditioning of the resulting stiffness systems, the incorporation of explicit conformality measures to strengthen the connection between energy minimization and true area minimization, and extensions to multipatch geometries with continuity constraints and adaptive refinement strategies.


\section*{Funding}
This research was supported by the Ministry of Higher Education Malaysia through the Fundamental Research Grant Scheme (FRGS/1/2023/STG06/USM/03/04) and by the School of Mathematical Sciences, Universiti Sains Malaysia.


\section*{Acknowledgements}
The authors are very grateful to the anonymous referees for their valuable suggestions.

\appendix
\section{Derivative identities for GT tensor-product patches}
\label{app:derivative-recursion}

This appendix records the derivative identities used in Section~\ref{sec:dirichlet-extremals}.  Our goal is to
justify the ``difference-form'' representations of the first partial derivatives, which are the key step for
separating the integrals and assembling the stiffness coefficients.  We keep the presentation compact, but we
include the main algebraic steps so that the argument reads as a proof rather than a direct computation.

\subsection{Univariate GT-B\'ezier basis: a derivative recursion}
Let $\{G_{k,n}(\cdot;\boldsymbol{\theta})\}_{k=0}^n$ be the univariate GT-B\'ezier basis generated from the seed
\eqref{eq:gt2} via the degree-elevation recursion \eqref{eq:generalbasis}.  Recall the recurrence (with the
standard convention $G_{k,n}\equiv 0$ for $k\notin\{0,\dots,n\}$):
\begin{equation}\label{eq:gt_rec}
G_{k,n}(t;\boldsymbol{\theta})
=(1-t)\,G_{k,n-1}(t;\boldsymbol{\theta})+t\,G_{k-1,n-1}(t;\boldsymbol{\theta}).
\end{equation}
Differentiating \eqref{eq:gt_rec} and using the product rule yields, for $n\ge 3$,
\begin{equation}\label{eq:gt_first_derivative_recursion_app}
\begin{aligned}
G'_{k,n}(t;\boldsymbol{\theta})
&=\frac{d}{dt}\Big((1-t)G_{k,n-1}(t;\boldsymbol{\theta})\Big)
 +\frac{d}{dt}\Big(t\,G_{k-1,n-1}(t;\boldsymbol{\theta})\Big) \\
&=-G_{k,n-1}(t;\boldsymbol{\theta})+(1-t)\,G'_{k,n-1}(t;\boldsymbol{\theta})
  +G_{k-1,n-1}(t;\boldsymbol{\theta})+t\,G'_{k-1,n-1}(t;\boldsymbol{\theta}),
\end{aligned}
\end{equation}
which is the claimed derivative recursion.  By iterating \eqref{eq:gt_first_derivative_recursion_app}, any
$G'_{k,n}$ reduces to the explicit derivatives of the seed functions.

\subsection{A univariate ``difference form'' for GT--B\'ezier curves}
The next lemma is the structural identity used later for surfaces: it rewrites the derivative of a GT curve in
terms of first differences of control points plus a lower-degree remainder.

\begin{lemma}[Derivative in difference form]\label{lem:curve_difference_form}
Let $F(t)=\sum_{k=0}^n G_{k,n}(t;\boldsymbol{\theta})\,b_k$ be a GT--B\'ezier curve of degree $n\ge 3$.
Then
\begin{equation}\label{eq:curve_difference_form_app}
F'(t)
=
\sum_{k=0}^{n-1}\Big(G_{k,n-1}(t;\boldsymbol{\theta})+t\,G'_{k,n-1}(t;\boldsymbol{\theta})\Big)\,(b_{k+1}-b_k)
+\sum_{k=0}^{n-1}G'_{k,n-1}(t;\boldsymbol{\theta})\,b_k .
\end{equation}
\end{lemma}

\begin{proof}
Differentiate $F$ and insert \eqref{eq:gt_first_derivative_recursion_app}:
\[
F'(t)=\sum_{k=0}^n G'_{k,n}(t)\,b_k
=\sum_{k=0}^n\Big[-G_{k,n-1}+(1-t)G'_{k,n-1}+G_{k-1,n-1}+tG'_{k-1,n-1}\Big]\,b_k.
\]
Group the terms into four sums and shift indices in the terms containing $k-1$:
\[
\sum_{k=0}^n G_{k-1,n-1}\,b_k=\sum_{k=0}^{n-1} G_{k,n-1}\,b_{k+1},\qquad
\sum_{k=0}^n t\,G'_{k-1,n-1}\,b_k=\sum_{k=0}^{n-1} t\,G'_{k,n-1}\,b_{k+1}.
\]
Using these shifts (and the convention $G_{n,n-1}\equiv 0$), we obtain
\[
F'(t)
=\sum_{k=0}^{n-1}\Big(G_{k,n-1}+tG'_{k,n-1}\Big)(b_{k+1}-b_k)
+\sum_{k=0}^{n-1}\Big((1-t)G'_{k,n-1}+tG'_{k,n-1}\Big)b_k,
\]
and the last bracket simplifies to $G'_{k,n-1}$, giving \eqref{eq:curve_difference_form_app}.
\end{proof}

\subsection{Tensor-product GT-B\'ezier surface: first partial derivatives}
Let $S$ be the GT tensor-product patch \eqref{eq:GT_surface_dirichlet} of bidegree $(m,n)$:
\[
S(u,v;\boldsymbol{\alpha},\mathbf{P})
=
\sum_{i=0}^{m}\sum_{j=0}^{n}
P_{ij}\,G_{i,m}(u;\boldsymbol{\alpha}^{(u)})\,G_{j,n}(v;\boldsymbol{\alpha}^{(v)}).
\]
Termwise differentiation immediately gives the separated representations
\begin{equation}\label{eq:SuSv_basic_app}
\begin{aligned}
S_u(u,v)
&=
\sum_{i=0}^{m}\sum_{j=0}^{n}
P_{ij}\,G'_{i,m}(u;\boldsymbol{\alpha}^{(u)})\,G_{j,n}(v;\boldsymbol{\alpha}^{(v)}),\\
S_v(u,v)
&=
\sum_{i=0}^{m}\sum_{j=0}^{n}
P_{ij}\,G_{i,m}(u;\boldsymbol{\alpha}^{(u)})\,G'_{j,n}(v;\boldsymbol{\alpha}^{(v)}),
\end{aligned}
\end{equation}
where $G'$ can be evaluated using \eqref{eq:gt_first_derivative_recursion_app}.

\subsection{Difference-form representations for \texorpdfstring{$S_u$}{Su} and \texorpdfstring{$S_v$}{Sv}}
For assembling the normal equations in Section~\ref{sec:dirichlet-extremals}, it is useful to isolate first
differences of control points in the $u$- and $v$-directions.  This follows by applying
Lemma~\ref{lem:curve_difference_form} in one parameter while treating the other parameter as a spectator.

Fix $v$ and define a $u$-curve with $v$-dependent control points
\[
B_i(v):=\sum_{j=0}^n P_{ij}\,G_{j,n}(v;\boldsymbol{\alpha}^{(v)}),
\qquad\text{so that}\qquad
S(u,v)=\sum_{i=0}^m G_{i,m}(u;\boldsymbol{\alpha}^{(u)})\,B_i(v).
\]
Applying Lemma~\ref{lem:curve_difference_form} to the curve $u\mapsto S(u,v)$ yields
\begin{equation}\label{eq:Su_diff_app}
S_u(u,v)
=
\sum_{i=0}^{m-1}\Big(G_{i,m-1}(u)+u\,G'_{i,m-1}(u)\Big)\,(B_{i+1}(v)-B_i(v))
+\sum_{i=0}^{m-1}G'_{i,m-1}(u)\,B_i(v),
\end{equation}
where, for readability, we suppressed the parameter vectors in $G_{i,\cdot}(\,\cdot\,;\boldsymbol{\alpha}^{(u)})$.
Expanding $B_{i+1}(v)-B_i(v)$ and $B_i(v)$ in terms of the surface control points gives
\begin{equation}\label{eq:Su_difference_form}
\begin{aligned}
S_u(u,v)
&=
\sum_{i=0}^{m-1}\sum_{j=0}^{n}
\Big( G_{i,m-1}(u;\boldsymbol{\alpha}^{(u)}) + u\,G'_{i,m-1}(u;\boldsymbol{\alpha}^{(u)}) \Big)
G_{j,n}(v;\boldsymbol{\alpha}^{(v)})\,(P_{i+1,j}-P_{i,j}) \\
&\quad+
\sum_{i=0}^{m-1}\sum_{j=0}^{n}
G'_{i,m-1}(u;\boldsymbol{\alpha}^{(u)})\,G_{j,n}(v;\boldsymbol{\alpha}^{(v)})\,P_{i,j}.
\end{aligned}
\end{equation}

The formula for $S_v$ is obtained analogously by fixing $u$ and applying
Lemma~\ref{lem:curve_difference_form} to the $v$-direction, which yields
\begin{equation}\label{eq:Sv_difference_form}
\begin{aligned}
S_v(u,v)
&=
\sum_{i=0}^{m}\sum_{j=0}^{n-1}
G_{i,m}(u;\boldsymbol{\alpha}^{(u)})
\Big( G_{j,n-1}(v;\boldsymbol{\alpha}^{(v)}) + v\,G'_{j,n-1}(v;\boldsymbol{\alpha}^{(v)}) \Big)\,(P_{i,j+1}-P_{i,j})\\
&\quad+
\sum_{i=0}^{m}\sum_{j=0}^{n-1}
G_{i,m}(u;\boldsymbol{\alpha}^{(u)})\,G'_{j,n-1}(v;\boldsymbol{\alpha}^{(v)})\,P_{i,j}.
\end{aligned}
\end{equation}

Formulas \eqref{eq:Su_difference_form}--\eqref{eq:Sv_difference_form} are exactly the representations invoked in
Section~\ref{sec:dirichlet-extremals} to separate the integrals and form the stiffness coefficients
\eqref{eq:I_defs}--\eqref{eq:J_defs}.

\bibliographystyle{unsrt}
\bibliography{00001}

\end{document}